\documentclass[11pt,a4paper,underruledhead]{paper}
\usepackage{enumerate,math}
\usepackage[mathscr]{euscript}

\title[Grassmannians of lines...]{Grassmannians of lines defined in the geometry\\ of a pseudo-polarity}
\author{K. Pra{\.z}mowski and M. {\.Z}ynel}

\DeclareMathOperator{\Rad}{Rad}
\DeclareMathOperator{\wspolin}{{\bf L}}
\DeclareMathOperator{\Hrd}{Hrd}
\DeclareMathOperator{\srodek}{{\bf q}}
\DeclareMathOperator{\GF}{GF}

\let\goth\mathfrak
\let\cal\mathcal
\let\field\mathbb

\def\LineOn(#1,#2){\overline{{#1},{#2}\rule{0em}{1,5ex}}}
\def\inc{\mathrel{\rule{2pt}{0pt}\rule[-0.8pt]{1pt}{2ex}\rule{2pt}{0pt}}}
\def\ninc{\mathrel{\not\mkern3mu\inc}}

\def\adjac{\mathrel{\sim}}

\def\penc{{\bf p}}
\def\pencr{\penc_{\rm r}}
\def\pencraf{\pencr^{\rm af}}

\def\TRG{\mathop{{\mbox{\boldmath$\Delta$}}}}

\def\lines{{\cal L}}
\def\planes{{\cal P}}
\def\planesr{\planes_{01}}
\def\peki{{\mathscr G}}
\def\pek(#1,#2){\penc({#1},{#2})}
\def\sub{{\mathscr P}}

\def\starof{\mathrm{S}}
\def\hipy{{\mathscr H}}
\def\reghipy{{\mathscr R}}
\def\regafhipy{{\mathscr A}}
\def\regbhipy{\mathscr A^\circ}

\def\linesr{\regafhipy_2}
\def\linesrb{\lines_{\rm r}}

\def\starofr{\starof_{\rm r}}
\def\starofraf{\starof_{\rm r}^{\rm af}}

\def\AffineSpSymb{\mathbf{A}}
\def\AfSpace(#1){\ensuremath{\AffineSpSymb(#1)}}

\def\KwadrSpSymb{\mathbf{Q}}
\def\KwadrSpace(#1,#2){\ensuremath{\KwadrSpSymb_{#1}(#2)}}
\def\AfPolSpSymb{\mathbf{U}}
\def\AfpolSpace(#1,#2){\ensuremath{\AfPolSpSymb_{#1}(#2)}}
\def\AfpolSpacex(#1,#2){\ensuremath{\AfPolSpSymb^\dagger_{#1}(#2)}}

\def\PencSpace(#1,#2){{\bf P}_{#2}({#1})}
\def\GrasSpace(#1,#2){{\bf G}_{#2}({#1})}
\def\AfPencSpace(#1,#2){{\bf A}_{#2}({#1})}

\def\fixaf{\ensuremath{\goth A}}
\def\fixafr{\ensuremath{\goth B}}
\def\fixafrb{\ensuremath{\goth C}}
\def\fixproj{\ensuremath{\goth P}}

\def\fixV{\ensuremath{\field V}}

\def\fixgras{\ensuremath{\goth G}}

\def\biegun{\mbox{\sf\bfseries b}}
\def\dhipa{\mbox{\sf\bfseries H}}
\def\shipa{\mbox{\footnotesize\sf\bfseries H}}
\def\sshipa{\mbox{\tiny\sf\bfseries H}}
\def\hipa{\mathchoice{\dhipa}{\dhipa}{\shipa}{\sshipa}}
\def\rdim{\mathrm{rdim}}

\def\PencilSp(#1,#2){{\bf P}_{#1}({#2})}

\def\TRGafrb{\mathop{\textstyle\TRG_{\fixafrb_2}}}
\def\TRGafr{\mathop{\textstyle\TRG_{\fixafr_2}}}

\def\div{\mathrel{\mid}}
\def\ndiv{\mathrel{\nmid}}

\newenvironment{cmath}{%
  \par
  \smallskip
  \centering
  $
}{%
  $
  \par
  \smallskip
  \csname @endpetrue\endcsname
}

\newenvironment{ctext}{%
  \par
  \smallskip
  \centering
}{%
 \par
 \smallskip
 \csname @endpetrue\endcsname
}

\begin{document}

\maketitle

\begin{abstract}
  The regular point-line geometry with respect to a pseudo-polarity
  is introduced. It is weaker than the underlying metric-projective geometry.
  The automorphism group of this geometry is determined.
  This geometry can be also expressed as the geometry of regular 
  lines and planes.
\end{abstract}

\begin{flushleft}\small
  Mathematics Subject Classification (2010): 51A50, 51F20, 51A45.\\
  Keywords: pseudo-polarity, regular subspaces, Grassmannian.
\end{flushleft}

\section*{Introduction}

Projective geometry and affine geometry over fields with even characteristic,
though have some ``strange" properties do not differ so strongly from geometries
over arbitrary field. In particular, most of standard methods used to characterize
geometry of the Grassmann spaces associated with them can be applied here.

The situation changes when we pass to the orthogonal geometry.
Let \fixproj\ be a projective space coordinatized by a field $\goth F$
and let $\perp$ be a projective conjugacy (projective polarity) defined on \fixproj;
then \fixproj\ equipped with $\perp$ is referred to as a {\em metric-projective space}.

Some standard derivatives are associated with the space $(\fixproj,\perp)$.
The first is a polar space,
which consists of selfconjugate (absolute) points and isotropic (singular or absolute) lines
(cf. \cite{tits}).
One of fundamental results which states that the underlying metric-projective space can
be reconstructed in terms of the associated polar space, is valid in all cases except
exactly one: when $\mathrm{char}({\goth F}) = 2$ and $\perp$ is a \emph{pseudo-polarity}.
In \cite[Subsection 2.4.18]{maldeghem} a polarity $\perp$ is said to be a pseudo-polarity 
if the set $\hipa$ of selfconjugate points with respect to $\perp$ is a proper (possibly empty)
subspace of $\fixproj$. In this paper we additionally assume that $\hipa$ is a hyperplane
(this means that ${\goth F}$ is perfect, see also \cite[Subsection 2.1.5]{hirs}). Pseudo-polarities are 
related to pseudo-quadrics and pseudo-quadratic forms 
(cf. \cite[Section 10.2]{buekenhout}, \cite[Chapter 6]{ueberberg}).
So, in the aforementioned point the geometry of a pseudo-polarity is
exceptional, though the geometry of its polar space is, generally, known. 

In a metric projective geometry an important role is played by the family of 
regular subspaces i.e. subspaces with the trivial radical.

If the conjugacy $\perp$ is not symplectic then the structure of regular 
points and regular lines is equivalent to the underlying metric geometry (cf. \cite{perpadeh}).
If $\perp$ is symplectic then one cannot have a space with regular points and lines. 
If $\perp$ is a pseudo-polarity then, admittedly $\mathrm{char}({\goth F}) = 2$, but
$(\fixproj,\perp)$ does contain regular points and lines. They yield
a new geometry a subject of our paper.
This geometry is stronger than
the affine geometry $\fixaf$ obtained by deleting nonregular points i.e. a projective
hyperplane $\hipa$ from $\fixproj$, but {\em weaker} than the
underlying metric-projective geometry in that
one can reconstruct the space \fixproj\ in terms of regular points and lines,
(cf. \ref{thm:reg2af}), but it is impossible to reconstruct the polarity $\perp$ 
in these terms (cf. \ref{thm:not:reg2perp:1}).
In this point the geometry of a pseudo-polarity is essentially exceptional.
The geometry of regular points and lines satisfies the dual of $\Gamma$-axiom 
(cf. \cite{cohen:handbook}), a relative of the $\Delta$-axiom 
(cf. \cite{copol:hall}, \cite{embcopol}). However, our geometry is neither one of copolar spaces, nor
a regular point-line geometry of a polarity (cf. \cite{grasregul}). Therefore,
results of these theories cannot be directly copied into our work.

In our paper we also discuss more  properties of the {\em regular geometry} of a 
pseudo-polarity, in particular, we determine the automorphism group of this geometry
(cf. \ref{prop:aut:nob3H}, \ref{prop:aut:b3H}).

The regular geometry of a pseudo-polarity, considered on the varieties of the
regular subspaces, presents some more oddities: the incidence system
of the regular subspaces contains isolated objects; there are also flags of regular
objects whose end-object may be contained in (or contains) no regular
successor (or predecessor respectively).
In geometry of incidence structures isolated objects generally are trash:
nothing can be defined in terms of them and they are unredefinable in general, as
each automorphism of the structure can arbitrarily permute them.
So, it is a good reason to remove them from considerations.
Since there are many combinations of isolated objects we get various 
structures of Grassmannians over regular subspaces.
Without entering into details of the general theory of $k$-Grassmannians
we show how our apparatus works in case of structures over regular lines.
It turns out (cf. \ref{thm:regb2af}, \ref{thm:sechip}, \ref{thm:regl2regp}) 
that all they are definitionally equivalent to the regular
point-line geometry. This, in particular, let us determine the automorphisms.
We are convinced that analogous results can be achieved for arbitrary $k$.

\section{Basic notions}

Let $\perp$ be a projective polarity in a finite dimensional metric projective space $(\fixproj,\perp)$
coordinatized by a vector space $\field V$ over a perfect field $\goth F$
with even characteristic.
In a suitable coordinate system the form $\xi$ which determines $\perp$
has its matrix of one of the following forms (cf. \cite[Subsection 2.1.5]{hirs}).
Let us set 
\begin{cmath}
  O = \begin{bmatrix} 0 & 0 \\ 0 & 0 \end{bmatrix}, \qquad
  \nabla = \begin{bmatrix} 0 & 1 \\ 1 & 0 \end{bmatrix}, \qquad\text{and}\qquad
  \nabla' = \begin{bmatrix} 1 & 1 \\ 1 & 0 \end{bmatrix}.
\end{cmath}
It is seen that $\det(\nabla),\det(\nabla')\neq 0$.
Then the matrix of $\xi$ is one of the following
\begin{center}
  \begin{tabular}{ccc}
    $\begin{bmatrix}
      1      & 0      & \cdots & \cdots & 0 \\       
      0      & \nabla & O      & \cdots & O \\       
      \vdots & O      & \nabla & \ddots & \vdots \\  
      \vdots & \vdots & \ddots & \ddots & O \\       
      0      & O      & \cdots & O      & \nabla     
    \end{bmatrix}$
    &
    $\begin{bmatrix} 
      \nabla' & O      & \cdots & O \\
      O       & \nabla & \ddots & \vdots \\
      \vdots  & \ddots & \ddots & O \\
      O       & \cdots & O      & \nabla
    \end{bmatrix}$
    &
    $\begin{bmatrix}
      \nabla & O      & \cdots & O \\
      O      & \nabla & \ddots & \vdots \\
      \vdots & \ddots & \ddots & O \\
      O      & \cdots & O      & \nabla
    \end{bmatrix}$
    \\
    {\refstepcounter{equation}\label{forma:typ1}
    Type \eqref{forma:typ1}}
    &
    {\refstepcounter{equation}\label{forma:typ2}
    Type \eqref{forma:typ2}}
    &
    {\refstepcounter{equation}\label{forma:typ3}
    Type \eqref{forma:typ3}}
  \end{tabular}
\end{center}

For a subspace $U$ we write $U^\perp = \set{p\colon p\perp U}$,
$\Rad(U) = U \cap U^\perp$, and $\rdim(U) = \dim(\Rad(U))$. 
A subspace $U$ is called {\em regular}
when $\Rad(U) = \emptyset$. The dimension $\dim(U)$ of a subspace $U$
coincides with the dimension of the vector subspace $W$ of $\field V$
such that $U = \set{\gen{u}\colon u \in W,\;u\neq\theta}$.
In particular, a point has dimension $1$ and a line has dimension $2$.
Frequently, a point $p$ will be identified with the one-element subspace $\set{p}$
which consists of $p$; consequently, we frequently write $U\cap U' = p$
instead of $U\cap U' = \set{p}$.
We write $\reghipy_k$ for the class of regular $k$-dimensional subspaces of 
$(\fixproj, \perp)$.
In the sequel we are interested in the incidence structure (cf. \cite{cohen:handbook})%
\begin{cmath}
  \reghipy := \left(\reghipy_k\colon 0\leq k \leq\dim(\fixproj)+1\right),
\end{cmath}
which is more closely investigated for indices $1, 2$ and $2, 3$. Various
geometries arise when we delete isolated objects in such structures.

\subsection{Grassmannians}

Grassmann spaces frequently appear in the literature, just to mention 
\cite{cohen:handbook}, \cite{polargras}.
The most general definition could be probably as follows. 
Let $X$ be a nonempty set and let $\sub$ be a family of subsets of $X$.
Assume that there is a dimension function $\dim:\sub\to\{0,\dots, n\}$ such that
$I=\struct{\sub, \subset, \dim}$ is an incidence geometry.
Write $\sub_k$ for the set of all $U\in\sub$ with $\dim(U) = k$. 
For $H\in\sub_{k-1}$ and $B\in\sub_{k+1}$ with $H\subset B$ the set
\begin{equation}\label{eq:pencil}
  \penc(H, B) := \set{U\in\sub_k\colon H\subset U\subset B}
\end{equation}
is called a \emph{pencil}; $\peki_k$ stands for the family of all such pencils.
If $0<k<n$, then a \emph{$k$-th Grassmann space over $I$} is a point-line geometry 
\begin{cmath}
  \PencSpace(I,k) := \struct{\sub_k, \peki_k}. 
\end{cmath} 
This is the most common understanding
of a Grassmann space. We used to call it a \emph{space of pencils} for its specific
lineset and to distinguish it from a closely related point-line geometry consiting
of $\sub_k$ as points and $\sub_{k+1}$ as lines, namely
\begin{cmath}
  \GrasSpace(I,k) := \struct{\sub_k, \sub_{k+1}, \subset}
\end{cmath}
which \emph{we} call a \emph{$k$-th Grassmannian over $I$} (cf. \cite[Section 1.1.2]{mark}).

In our settings, we write $\hipy_k$ for the set of $k$-dimensional subspaces of \fixproj.
Then a \emph{projective pencil} is a set \eqref{eq:pencil} with $\sub = \hipy$ and
a \emph{regular pencil} is a set \eqref{eq:pencil} with $\sub = \reghipy$.

\subsection{Methodological issues}

Most of the results of our paper consists in various
definablilities/undefinabilities of particular structures, derived from the
incidence geometry $\reghipy$, in terms of other structures of the same type.
Clearly, in a structure like this various notions can be introduced and some
notions can originate in the underlying projective and metric projective
geometry. But proving our results we must be strict: in each particular case we
must be sure that the respective definition can be expressed entirely in terms
of the language with names only for primitive notions of the structure in which
the second one is defined. The safest way to ensure that is to write down
respective definitions as ``formal formulas" in a formal language, and we do
follow this convention.

Our structures are, primarily, point-line geometries and the only primitive
notion used to characterize their geometry is the \emph{incidence} relation. 
We use the symbol $\inc$ to name the incidence relation between points and lines.
Note that $\mathord{\inc}=\mathord{\in}$ 
in $ \PencSpace(I,k)$ while $\mathord{\inc}=\mathord{\subset}$ in $\GrasSpace(I,k)$.

By a \emph{triangle} in a point-line geometry we mean three points, called \emph{vertices}, 
and three lines, called \emph{sides}, where every vertex is incident to exactly two sides
(or dually, every side is incident to exactly two vertices).

\subsection{$\perp$ is symplectic}

The form is of type \eqref{forma:typ3}.
Consequently, $\dim(\fixV) = 2m$ for some integer $m$.
In this case each point is selfconjugate; the point set of the respective 
polar space is the point set of \fixproj.

\subsection{$\perp$ is a pseudo-polarity}\label{sec:pseudo-polarity}

The form is of type \eqref{forma:typ1} or \eqref{forma:typ2}.
Then the set of selfconjugate points under $\perp$ is 
a subspace $\hipa$ of $\fixproj$. We assume that $\hipa$ is a hyperplane; then
the pole  $\biegun$ of $\hipa$ is a point.
The set $\reghipy_1$ of regular points is the hyperplane complement: the complement of
$\hipa$.
One can imagine the geometry of regular points and regular lines as a fragment
of an affine geometry.

The restriction $\perp_{\hipa}$ of 
the conjugacy $\perp$ to $\hipa$ determines on $\hipa$ 
a (possibly degenerate) symplectic polarity. 
If $\Rad(\hipa) = \emptyset$,  this polarity is nondegenerate.
In general, $\dim(\Rad(\hipa))\leq 1$.

Again, two types of geometry may occur.

\subsubsection{$\biegun$ lies on $\hipa$}\label{subsubsec:b0H}

In this case $\Rad(\hipa) = \biegun$, so the polarity $\perp$ restricted to
$\hipa$ is a degenerate symplectic polarity.
Then $n = \dim(\fixV) = 2k$ for some integer $k$
and
the form $\xi$ has form \eqref{forma:typ2}; it may be written as 
\begin{cmath}
  \xi([x_0,x_1,\dots,x_{2k-1}],[y_0,y_1,\dots,y_{2k-1}])  =
  x_0 y_0 + \sum_{i=0}^{k-1} (x_{2i} y _{2i+1} + x_{2i+1} y _{2i})
\end{cmath}
(cf. \cite{hirs}).
The hyperplane $\hipa$ and its pole $\biegun$ are characterized by the equations
\begin{cmath}
  \hipa\colon x_0 = 0, \qquad \biegun = [0,1,0,\dots,0].
\end{cmath}

\subsubsection{$\biegun$ does not lie on $\hipa$}

Then $n = \dim(\fixV) = 2k + 1$ 
and
the form $\xi$ has form \eqref{forma:typ1}; it may be written as 
\begin{cmath}
  \xi([x_0,x_1,\dots,x_{2k}],[y_0,y_1,\dots,y_{2k}])  =
  x_0 y_0 + \sum_{i=1}^{k} (x_{2i-1} y _{2i} + x_{2i} y _{2i-1})
\end{cmath}
(cf. \cite{hirs}).
The hyperplane $\hipa$ and its pole $\biegun$ are characterized by the equations
\begin{cmath}
  \hipa\colon x_0 = 0, \qquad \biegun = [1,0,\dots,0].
\end{cmath}

Note that for $x \in \hipa$ we have $x^\perp = \biegun + (x^\perp\cap\hipa)$.
Therefore, to determine $x^\perp$ it suffices to know the restriction $\perp_{\hipa}$ of 
the conjugacy $\perp$ to $\hipa$. On the other hand, this restriction determines on
$\hipa$ a symplectic polarity and this polarity is nondegenerate as 
$\Rad(\hipa) = \emptyset$.

\section{Results, general}

Let $\perp$, $\hipa$, and $\biegun$ be like in Section~\ref{sec:pseudo-polarity}.
Two families of specific regular subspaces arise:
\begin{cmath}
  \regafhipy_k = \{ A\in\reghipy_k\colon A\not\subset\hipa \}, \qquad\qquad
  \regbhipy_k = \{ A\in\reghipy_k\colon \biegun\notin A \}.
\end{cmath}
If $A$ is a subspace of $\fixproj$ we write $A^\infty := A\cap\hipa$
for the set of selfconjugate points of $A$. Note that $\Rad(A)$ is
always a (possibly empty) set of selfconjugate points, i.e. 
$\Rad(A)\subset A^\infty\subset\hipa$.

We begin with a simple but very significant fact.
\begin{fact}[{\cite[Cor. 1.3]{grasregul}}]\label{fct:reg:hip}
  If a subspace $A$ contains a regular hyperplane or $A$ is a hyperplane
  of a regular subspace, then $\rdim(A)\leq 1$.
\end{fact}

\begin{fact}\label{fct:dimkrosperp}
  If\/ $A$ is a subspace not contained in $\hipa$, then 
  $\dim(A\cap(A^\infty)^\perp)\geq 1$.
\end{fact}
\begin{proof}
  Set $k := \dim(A)$, $m := \dim(A\cap(A^\infty)^\perp)$.
  Then $\dim(A^\infty) = k-1$ and 
  $n \geq \dim(A + (A^\infty)^\perp) = k + (n - (k -1)) - m = n+1 - m$,
  which yields our claim.
\end{proof}
The subspace $A\cap(A^\infty)^\perp$ will be denoted by $\Hrd(A)$
and will play an essential role (slightly similar to the role of the 
radical $\Rad(A)$ of $A$). Fact \ref{fct:dimkrosperp} says that $\Hrd(A)$ is at least 
a point for every subspace $A$ not contained in $\hipa$.

\begin{prop}\label{prop:reg:sub}
  Let\/ $A$ be a subspace of\/ $\fixproj$ not contained in $\hipa$.
  The following conditions are equivalent:
  \begin{enumerate}[\rm(i)]
  \item\label{regsub:war1}
    the subspace $A$ is regular; 
  \item\label{regsub:war2}
    the subspace $\Hrd(A)$ is a point (cf. {\upshape \ref{fct:dimkrosperp}});
  \item\label{regsub:war3}
    $\rdim(A^\infty)\leq 1$ and 
    \begin{enumerate}[\rm(a)]
    \item
      either the subspace $A^\infty$ is regular in the (possibly degenerated) 
      symplectic projective geometry induced on $\hipa$,
    \item
      or $\Rad(A^\infty)$ is a point $p$ and $A\not\subset p^\perp$.
    \end{enumerate}
  \end{enumerate}
\end{prop}
\begin{proof}
  From \ref{fct:dimkrosperp}, $\Hrd(A)$ is at least a point.
  
  \eqref{regsub:war1}$\implies$\eqref{regsub:war3}:
  Since $A^\infty$ is a hyperplane in $A$, by \ref{fct:reg:hip},
  $\rdim(A^\infty)\leq 1$.
  If $\rdim(A^\infty) = 0$, then $A^\infty$ is regular.
  If $\rdim(A^\infty) = 1$, then $\Rad(A^\infty)$ is a point, say $p$.
  The point $p$ could be the only point in $\Rad(A)$.
  As $A$ is regular, we have $p\notin\Rad(A)$ i.e. $p\not\perp A$ or,
  in other words, $A\not\subset p^\perp$.
  
  \eqref{regsub:war3}$\implies$\eqref{regsub:war2}:
  Let $A^\infty$ be regular. Suppose that $\Hrd(A)$ contains a line $K$.
  Then $K\subset A$. As $A^\infty$ is a hyperplane in $A$ so, $K$ and $A^\infty$
  share a point $q$. Note that $q\in A^\infty\cap(A^\infty)^\perp = \Rad(A^\infty)$
  which is impossible. Consequently, $\Hrd(A)$ is a point.
  
  Now, let $\Rad(A^\infty)$ be a point $p$ and $A\not\subset p^\perp$.
  Then $p\in\Hrd(A)$. 
  Suppose that there is $q\in\Hrd(A)$, $q\neq p$. 
  If $q\in A^\infty$, then $q\in\Rad(A^\infty)$, which is impossible. 
  So $q\in A\setminus A^\infty$. On the other hand $p\in A^\infty\cap(A^\infty)^\perp$ and 
  $q\in(A^\infty)^\perp$, hence $p\perp A^\infty + q = A$ which contradicts
  our assumption that $A\not\subset p^\perp$. So, $\Hrd(A)$ is a point.
  
  \eqref{regsub:war2}$\implies$\eqref{regsub:war1}:
  Let $\Hrd(A)$ be a point $p$ and suppose that $q\in\Rad(A)$. Then $q\in A$ and 
  $q\in A^\perp\subset(A^\infty)^\perp$, so $q = p$.
  Hence $p\perp A$ and $p\in A^\infty$, which gives $p \perp A + (A^\infty)^\perp$. 
  Note that $\dim(A+(A^\infty)^\perp) = n$, so $p\perp V$ a contradiction.
\end{proof}
\begin{note*}
  If\/ $p = \Rad(A^\infty)$, then  $p = \Hrd(A)$.
\end{note*}
\begin{note*}
  Let\/ $A$ be a regular subspace not contained in $\hipa$ and $k = \dim(A)$. 
  Then the restriction $\perp_A$ of $\perp$ to $A$ is a pseudo-polarity and 
  $\Hrd(A)$ is the pole of $A^\infty$ within $A$.
  Thus
  \begin{ctext}  
    if $2\div k$, then $\Hrd(A)\in A^\infty$,\qquad and\qquad
    if $2\ndiv k$, then $\Hrd(A)\notin A^\infty$.
  \end{ctext}
\end{note*}

Recall that the induced geometry on $\hipa$ is symplectic.
Therefore, $A^\infty$ may be regular iff 
$2\mathrel{\div}\dim(A^\infty)$ i.e. iff $2\ndiv\dim(A)$.
This gives

\begin{cor}\label{cor:reg:sub}
  Let\/ $A$ be a subspace not contained in $\hipa$ with $k = \dim(A)$.
  \\
  If\/ $2\div k$ then $A$ is regular iff\/ $\Rad(A^\infty)$ is a point $p$ and $A\setminus \hipa$ misses 
  $p^\perp$.
  \\
  If\/ $2\ndiv k$ then $A$ is regular iff\/ $A^\infty$ is regular.
\end{cor}

As particular instances of \ref{prop:reg:sub} we obtain a series of 
criterions of regularity.

\begin{lem}\label{lem:reg:line}
  A line $L$  not contained in $\hipa$ is regular iff 
  it is not contained in the hyperplane $(L^\infty)^\perp$.
\end{lem}

\begin{cor}\label{cor:regularneWpeku}
  Let $p$ be a regular point.
  A line $L$ through $p$ is regular iff it misses $p^\perp\cap\hipa$.
\end{cor}
\begin{proof}
  Let $q := L \cap \hipa$. By \ref{lem:reg:line} $L$ is regular iff $p\notin q^\perp$
  i.e. iff $q\notin p^\perp$, as required.
\end{proof}

Note that every nonregular line on $\hipa$ is totally isotropic.

\begin{lem}\label{lem:regul3biegun}
  If\/ $\biegun\in\hipa$, then each affine line with direction $\biegun$
  is regular and no line on $\hipa$ through $\biegun$ is regular. 
  If\/ $\biegun\notin\hipa$, then no affine line through\/ $\biegun$ is regular.
\end{lem}
\begin{proof}
  Let $\biegun\in\hipa$.
  In view of \ref{lem:reg:line} it suffices to note that $\biegun^\perp = \hipa$
  and affine line are those not contained in $\hipa$.

  Let $\biegun\notin\hipa$. Then for each $p\in\hipa$ we have $p\in\Rad({\LineOn(p,\biegun)})$
  so, the line $\LineOn(p,\biegun)$ is not regular.
\end{proof}

A direct consequence of \ref{prop:reg:sub} and \ref{cor:reg:sub} is
\begin{lem}\label{lem:reg:plane}
  Let $A$ be a plane not contained in $\hipa$.
  Clearly, $A^\infty$ is a line.
  The following conditions are equivalent:
  \begin{enumerate}[\rm(i)]
  \item\label{regplane:war1}
    the plane $A$ is regular; 
  \item\label{regplane:war2}
    the subspace $\Hrd(A)$ is a point (cf. {\upshape \ref{fct:dimkrosperp}});
  \item\label{regplane:war3}
    the line $A^\infty$ is regular in the (possibly degenerated) symplectic projective geometry 
    induced on $\hipa$.
  \end{enumerate}
\end{lem}

\begin{lem}\label{lem:planeonhipa}
  Let $A$ be a plane contained in $\hipa$. Then either
  \begin{sentences}
  \item
    $\Rad(A)$ is a point and all non-regular lines on $A$ form a pencil through that point, or
  \item 
    $\Rad(A)$ is a line and no line on $A$ is regular.
  \item
    $\Rad(A)$ is the entire plane $A$, i.e. $A$ is totally isotropic.  
  \end{sentences}
\end{lem}

\begin{proof}
  It is clear that $1\le\rdim(A)$ as $\dim(A)=3$ and $A\subseteq\hipa$. 
  
  In case $p := \Rad(A)$ is a point
  then $A\subseteq p^\perp$, so $p$ is the radical of every line on $A$ through $p$.
  Note that if $q$ is the radical of some line on $A$ not through $p$, then 
  $A\subseteq q^\perp$ and we would have $\LineOn(p,q) = \Rad(A)$ which is impossible.
  
  Now, if $L := \Rad(A)$ is a line, then every line of $A$ crosses $L$ and thus
  is non-regular.
\end{proof}

\begin{rem*}
  Let $A$ be a plane not contained in $\hipa$. If\/ $\Hrd(A)$
  is a point $p$ not on $\hipa$ then $A$ is regular, but no line through
  $p$ contained in $A$ is regular.
\end{rem*}

Further we assume that:
\begin{ctext}\itshape
  lines of\/ \fixproj\ are of size at least 6,
\end{ctext}
which means that the ground field of \fixproj\ is not $\GF(2)$ and not $\GF(4)$.
Most of our reasonings remain true for $\GF(4)$ and those few which fail will be
indicated.

\section{Grassmannians of regular points and lines}

\subsection{Regular point-line geometry} 

In what follows we shall pay attention to the 
Grassmannian of regular points and lines of the pseudo-polarity $\perp$, namely
\begin{cmath}
  \fixgras_1 := \GrasSpace(\reghipy,1) = \struct{\reghipy_1,\reghipy_2,\subset}.
\end{cmath}
Observe, first, that the set $\reghipy_1$ is simply the point-complement of
the hyperplane $\hipa$. Let 
$\fixaf = \struct{\reghipy_1,\lines}$ be the affine space obtained from
\fixproj\ by deleting the hyperplane $\hipa$;
then $\regafhipy_2\subset\lines$.

\begin{fact}\label{fct:isol1}
  If\/ $\biegun\in\reghipy_1$, then $\biegun$ is an isolated point in $\fixgras_1$.
  If\/ $L\in\reghipy_2$ and $L\subset\hipa$, then $L$ is an isolated line in $\fixgras_1$.
\end{fact}

\begin{proof}
  The first statement restates \ref{lem:regul3biegun}, while the other is trivial. 
\end{proof}

Let $\fixafr_1$ be the structure obtained from $\fixgras_1$
by deleting its isolated points and  lines.
Then
\begin{cmath}
  \fixafr_1 = \struct{\regbhipy_1,\linesr,\in}
\end{cmath}
Note that $\fixafr_1$ is a substructure of \fixaf.
This structure, primarily, and related structures will be investigated in this section.

A plane $A$ of $\fixproj$ not contained in $\hipa$ is a plane of $\fixaf$ and
will be  referred to as an affine plane; $A^\infty$ is the set of its improper
points. 
Similarly, a line $L$ of $\fixproj$  not contained in $\hipa$ is a line of
$\fixaf$ and will be referred to as an affine line; $L^\infty$ is its improper
point.

So, let $A$ be an affine plane and set $L := A^\infty$.
\begin{fact}
  Through each $q\in L$ there passes a nonregular affine line $M$ contained in $A$
  (in every direction in $A$ there is a nonregular line contained in $A$).
\end{fact}
\begin{proof}
  It suffices to consider $A\cap q^\perp$, which is at least a line.
\end{proof}
\begin{fact}
  If $A$ contains two parallel nonregular lines then $A$ is nonregular as well.
\end{fact}
\begin{proof}
  Let $M_1\parallel M_2$ be nonregular, $M_1,M_2 \subset A$. Take $q = M_1^\infty$.
  Then $q \perp M_1,M_2$, so $q \perp M_1 + M_2 = A$ i.e. $q \in\Rad(A)$.
\end{proof}
\begin{fact}
  If $A$ contains a triangle with all its sides nonregular then
  $\Rad(A) = A^\infty$ and no affine line contained in $A$ is regular.
\end{fact}
\begin{proof}
  Let $a_1,a_2,a_3$ be the vertices of a required triangle and 
  $q_i = {\LineOn(a_j,a_k)}^\infty$ for $\set{i,j,k} = \set{1,2,3}$. 
  Then 
  $a_i\perp q_j,q_k$, so $a_i \perp A^\infty$ for each $i$.
  Thus $A\perp A^\infty$.
\end{proof}

\begin{fact}\label{fct:classif:planes}
  The following possibilities may occur.
  \begin{description}
  \item[$L$ is nonregular:]
    Then $L$ is isotropic i.e. $L\perp L$, $A$ is nonregular, and we have two cases.
    \begin{description}
    \item[$\Rad(A)$ is the line $L$:]
      Then $L\perp A$ and $A$ contains no regular line.
      In this case $\Hrd(A) = A$.
    \item[$\Rad(A)$ is a point $q$:]
      Then $q\in L$. An affine line $K$ on $A$ is nonregular iff $K^\infty = q$.
      In this case $\Hrd(A) = L$.
    \end{description}
  \item[$L$ is regular:]
    Then $A$ is regular and $\Hrd(A)$ is an affine point $p$ on $A$. 
    An affine line $K$ on $A$ is nonregular iff $p\in K$.
  \end{description}
\end{fact}

\begin{proof}
  The claim is nearly evident. It only remains to prove the above characterization of 
  nonregular lines on $A$. Let $K$ be an affine line on $A$ and $x = K^\infty$.
  \par
  Let $A\perp L$; then $L \ni x \perp A \supset K$ and thus $K$ is not regular.
  \par
  Let $\Rad(A) = q$ and $L\perp L$; clearly, $q \in L$. Let $q \in K$; as above, 
  $q \perp A \supset K$ and thus $q \perp K$. Let $q \notin K$ and suppose that  $K$
  is not regular. Then $x \neq q$ and $x \perp L + K = A$, which gives 
  $A \perp q + x = L$. The obtained contradiction yields that $K$ is regular.
  \par
  Let $A$ be regular. From definition, $p \perp L$ and thus $p \perp x$; with $x \perp x$
  from $p \in K$ we obtain $x \perp K$, so $K$ is not regular.
  Assume that $K$ misses $p$ and suppose that $K$ is not regular.
  Then $K$ and $\LineOn(x,p)$ are two parallel nonregular lines on $A$
  and thus $A$ is not regular. This contradiction yields that $K$ must be regular.
\end{proof}
As a consequence we get
\begin{fact}\label{fct:plane:srodek}
  Let $A$ be an affine plane with\/ $\rdim(A)\leq 1$.
  In view of\/ {\upshape \ref{fct:classif:planes}} the nonregular lines on $A$ form a pencil.
  Its vertex will be denoted by $\srodek(A)$. This pencil is
  \begin{itemize}\def\labelitemi{\null}
  \item
    proper and\/ $\srodek(A)$ is an affine point when $\rdim(A)=0$, or
  \item
    parallel and\/ $\srodek(A)$ is a point on $A^\infty$ when $\rdim(A)=1$. 
  \end{itemize}
\end{fact}
\begin{lem}
  Let $M_1,M_2$ be two parallel regular lines in $\lines$ and let $A$ be the affine 
  plane that contains them. Then either $A$ is regular or $\Rad(A)$ is a point.
  In both cases $A$ contains a pair $K_1,K_2$ of regular lines which intersect
  in an affine point such that $K_i$ crosses $M_j$ in an affine point for all $i,j$.
\end{lem}
\begin{proof}
  Let $a_1$ be any point on $M_1$ and $K_0$ be the unique nonregular line through $a_1$
  (in above notation, either $K_0 = \LineOn(a_1,q)$ or $K_0 = \LineOn(a_1,p)$, resp.).
  Let $y = M_2 \cap K_0$ and $a_2$ be an affine point on $M_2$ distinct from $y$;
  take $K_1 = \LineOn(a_1,a_2)$.
  Let $b \in K_1$ be an affine point distinct from $a_1,a_2$ and $K_2 = \LineOn(b,y)$.
\end{proof}

\begin{cor}
  The formula
  \begin{multline}
    M_1 \parallel M_2 \iff (\exists{K_1,K_2})(\exists{p,a_1,a_2,b_1,b_2})
     \big[ K_1\neq K_2 \Land p \inc K_1,K_2
     \\
     a_1 \inc K_1,M_1 \Land a_2 \inc K_1,M_2 \Land b_1 \inc K_2,M_1 \Land b_2 \inc K_2,M_2
     \Land p \ninc M_1,M_2\big]
     \\
     \Land \neg\exists{a}[a \inc M_1,M_2]
  \end{multline}
  defines the parallelism of regular lines in terms of the geometry of\/ $\fixafr_1$.
\end{cor}

\begin{lem}\label{lem:regtriang}
  Let $A$ be an affine plane with\/ $\rdim(A)\leq 1$.
  Then $A$ contains a triangle $\Delta$ with regular sides
  and, moreover, through each affine point on $A$ distinct from $\srodek(A)$
  there passes a regular line
  which crosses the sides of $\Delta$ in at least two affine points.
\end{lem}
\begin{proof}
  In view of \ref{fct:plane:srodek} the nonregular lines on $A$ form a pencil
  with the vertex $\srodek(A)$:
  a proper one if $\srodek(A)$ is an affine point $p$ or a parallel one when $\srodek(A)$
  is a point $q$ on $A^\infty$. Thus the existence of a required
  triangle $\Delta$ is evident. Let $a_1,a_2,a_3$ be the vertices of $\Delta$
  and $x$ be an arbitrary affine point on $A$. The only lines 
  through $x$
  that may not cross appropriately the sides of $\Delta$ are the following
  $\LineOn(x,a_1) \parallel \LineOn(a_2,a_3)$
  $\LineOn(x,a_2) \parallel \LineOn(a_1,a_3)$,
  $\LineOn(x,a_3) \parallel \LineOn(a_1,a_2)$, and
  $\LineOn(x,{\srodek(A)})$. 
  (Note that from the Fano axiom valid in \fixaf, the lines
  through $a_i$ parallel to $\LineOn(a_j,a_l)$, $\set{i,j,l} = \set{1,2,3}$ 
  have indeed a common point $x$.)
  From assumptions, there are at least 5 lines through $x$ contained in $A$ and thus
  there exists a line required as well.
\end{proof}

In $\fixafr_1$ for a triangle $\Delta$ with the sides $L_1,L_2,L_3$ we define
\begin{multline}\label{eq:fullplane}
  \pi(\Delta) := \big\{
  x \colon (\exists{K})(\exists{a,b})\big[ x,a,b \inc K \Land a \neq b \Land
  \\
  \big((a \inc L_1 \land b \inc L_2)\lor (a \inc L_2 \land b \inc L_3)
  \lor (a\inc L_1 \land b \inc L_3)\big) \big] \big\}.
\end{multline}

Let $\planes$ be the set of planes of\/ \fixaf,   and let $\planes_i$ be the set
of planes in $\planes$ with $\rdim = i$. Then $\planesr := \planes_0
\cup\planes_1$ is the set of  planes with $\rdim\leq 1$.
For $A\in\planesr$ write $[A] := A\setminus\set{\srodek(A)}$.

\begin{cor}\label{cor:reconplanes}
  If $A\in\planes_1$, then $[A] = A$. If $A\in\planes_0$, then
  $[A]$ is an affine plane with one point deleted.
  Moreover, we have
  \begin{equation}
    \big\{[A]\colon A \in\planesr\big\} = 
    \big\{ \pi(\Delta)\colon \Delta \text{ is a triangle in }\fixafr_1\big\}.
  \end{equation}
\end{cor}

\begin{lem}\label{lem:planes:przez:line}
  Let $L$ be a nonregular affine line through an affine point $p$ in $\fixaf$.
  Then $q := L^\infty \neq\biegun$.
  Let $A\in \planesr$ contain $L$ and $M := A^\infty$, so $A = L + M$ with
  $q\in M\subset \hipa$.
  Then one of the following holds.
  \begin{sentences}\itemsep-2pt
  \item\label{planes:przez:line:1}
    $M\cap q^\perp = q$ ($M \not\subset q^\perp$). 
    In this case $M$ is regular, so
    $A$ is regular as well. Clearly, $\srodek(A)\in L$, so $M \subset {\srodek(A)}^\perp$.
    To have $p\neq \srodek(A)$ we need $M\not\subset p^\perp$.
  \item\label{planes:przez:line:2}
    $M\subset q^\perp$.
    In this case either $M\not\perp L$ or $\biegun$ is an affine point on $L$.
  \end{sentences}
\end{lem}
\begin{proof}
  Case \eqref{planes:przez:line:1} is evident.
  If $q\in M\subset q^\perp$, then $M\perp M$. To have $M\neq\Rad(A)$ we must 
  have $x\not\perp M$ for each $x\in A \setminus M$; this is equivalent to
  $L\not\perp M$. 
  Assume that $L\perp M$. Then $M=\Rad(A)$ and $L\perp q^\perp\cap\hipa$.
  Comparing dimensions we get $L^\perp = q^\perp\cap\hipa$
  which gives $L = q + \biegun$.
\end{proof}

\begin{lem}\label{lem:reg2collin}
  Let $L$ be a nonregular line and $a_1,a_2,a_3$ be affine points on $L$ in $\fixaf$.
  If\/ $\biegun\neq a_1,a_2,a_3$, 
  then there are distinct $A_1,A_2\in\planesr$ with $a_1,a_2,a_3 \in [A_1]\cap[A_2]$.
\end{lem}
\begin{proof}
  In view of \ref{lem:planes:przez:line} it suffices to find two lines 
  $M_1,M_2\subset\hipa$ through $q = L^\infty$ such that 
  $M_1,M_2 \not\subset q^\perp,a_1^\perp,a_2^\perp,a_3^\perp$.
  The required lines exist as we have at least 6 lines in a projective pencil.
\end{proof}

\begin{thm}\label{thm:reg2af}
  The affine space \fixaf\ is definable in terms of\/ $\fixafr_1$ and, 
  consequently, \fixproj\ is definable in terms of\/ $\fixafr_1$ as well.
\end{thm}
\begin{proof}
  Let $\wspolin_0$ be the ternary collinearity relation 
  on the set of affine points distinct from $\biegun$ on nonregular lines.
  By \ref{lem:reg2collin} this relation
  is definable in terms of the geometry of $\fixafr_1$.
  On the other hand $\wspolin_0$ is a ternary equivalence relation 
  in the sense of \cite[\S 4.5]{bachman}.   
  So, let $\lines_0$ be the family of equivalence classes of $\wspolin_0$.
  If $\biegun$ is an improper point we are through, as $\lines_0$ consists of 
  the affine nonregular lines.
  So, assume that $\biegun$ is an affine point. 
  Then $\lines_0$ consists of the nonregular lines not through $\biegun$ and the sets
  $L\setminus\set{\biegun}$ where $L$ is nonregular through $\biegun$.
  For every triangle $\Delta$ of $\fixafr_1$ we set
  \begin{cmath}
    \pi'(\Delta) = 
    \bigcup\big\{ L\in\linesr\cup\lines_0\colon |\pi(\Delta)\cap L|\geq 2\big\}.
  \end{cmath}
  Let $L\in\lines_0$ and $L'$ be the nonregular line of $\fixaf$ that contains $L$.
  Either $\biegun\in L'$ and then every plane containing $L$
  also contains $\biegun$, or $\biegun\notin L'$ and then there is exactly one plane
  $L+\biegun$ containing $L$ and $\biegun$.
  Therefore,
  if no plane through a line $L$ is an affine plane, then $L\in\lines_0$ and we set 
  $L' := L\cup\{\biegun\}$; otherwise we set $L' := L$.
  After that
  $\lines'_0:=\set{L'\colon L\in\lines_0}$ is the set of all nonregular lines of \fixaf\
  and $\lines = \linesr \cup\lines'_0$.
\end{proof}

For a set $X$ of affine points we write $\overline{X}$
for the least subspace of \fixproj \ which contains $X$.

\begin{lem}\label{lem:reg2horiz}
  Let $q \in\hipa$. Then the set
  \begin{equation}
    [q] := \{ a\colon a \text{ is a point of\/ }\fixaf,\;\LineOn(a,q)\notin\linesr \}
  \end{equation}
  is the set of affine points on $q^\perp$, and thus
  $\overline{[q]} = q^\perp$.
\end{lem}

Similarly, if $a$ is a point of \fixaf\ then 
the set 
$[a] = \big\{ K^\perp\colon a \inc K,\; K\in\lines\setminus\linesr \big\}$
coincides with the set $a^\perp \cap\hipa$; but not with the set $a^\perp$,
unhappily.

\begin{thm}\label{thm:not:reg2perp:1}
  The metric projective space $(\fixproj,\perp)$ is not definable in terms of
  the geometry of\/ $\fixafr_1$.
\end{thm}
\begin{proof}
  Let $W$ be the subspace of $\field V$ with
    $\hipa = \set{\gen{h}\colon h\in W,h\neq \theta}$, 
  and 
  $\xi$ be the form defined on $\field V$ which determines the conjugacy $\perp$.
  Write $\biegun=\gen{e_0}$ for a vector $e_0$.
  There are two cases to consider.
  \begin{sentences}\itemindent13ex
  \item[$\biegun\notin\hipa$:]
    Let $\xi_0$ be the restriction of $\xi$ to $W$.
    Then  $\xi_0$ is a nondegenerate symplectic form. 
    Write $\varepsilon = \xi(e_0,e_0)$.
    Let $e_1,\dots,e_n$ be a basis of $\hipa$; then the family
    ${\mathscr E} = \left( e_i\colon i=0,\dots,n \right)$
    is a basis of $\field V$.
    We have $e_0 \perp h$ for every $h\in\hipa$ and thus
    the formula defining the form $\xi$ is the following
    \begin{equation}\label{equ:form1}
      \xi(h_1 + \alpha_1 e_0,h_2 + \alpha_2 e_0) = \xi_0(h_1,h_2) + \alpha_1 \alpha_2 \varepsilon,
    \end{equation}
    where $h_1,h_2 \in W$ and $\alpha_1,\alpha_2$ are scalars of the coordinate field.
    Note that, conversely, for every nondegenerate symplectic form  $\xi_0$ defined on $W$
    and every scalar $\varepsilon\neq 0$ the formula \eqref{equ:form1}
    defines a nondegenerate bilinear form $\xi = \xi_\varepsilon$. 
    Indeed, if $M$ is the matrix of $\xi_0$ in the given basis then
    \begin{cmath}
      M_\varepsilon = 
        \begin{bmatrix} 
          \varepsilon & 0 & \cdots & 0 \\
          0 \\
          \vdots      &   & M \\ 
          0
        \end{bmatrix}
    \end{cmath}
    is the matrix of $\xi_\varepsilon$ and $\det(M_\varepsilon)\neq 0$.
    In particular, for $h,h_1\in W$ and a scalar $\alpha$ we have
    \begin{cmath}
      \xi_\varepsilon(h_1,h+ \alpha e_0) = \xi_0(h_1,h)
      \qquad\text{and}\qquad
      \xi_\varepsilon(h+\alpha e, h + \alpha e_0) = \alpha^2 \varepsilon.
    \end{cmath}
    Let us write $\perp_\varepsilon$ for the conjugacy determined by $\xi_\varepsilon$
    and ${\fixafr}_{\varepsilon}$ for the induced structure of regular points and 
    lines wrt. the conjugacy $\perp_{\varepsilon}$.
    Let $\varepsilon_1,\varepsilon_2$ be any two nonzero scalars. 
    Then the following holds
    \begin{eqnarray}\label{equ:form1a}
      q' \perp_{\varepsilon_1} q'' & \iff & q' \perp_{\varepsilon_2} q'',
      \\ \label{equ:form1b}
      q' \perp_{\varepsilon_1} a & \iff & q' \perp_{\varepsilon_2} a,
    \end{eqnarray}
    for all $q',q''\in\hipa$, $a\notin\hipa$.
    From \eqref{equ:form1}, \eqref{equ:form1a} and \eqref{equ:form1b} we derive that 
    the set of points selfconjugate under $\xi_{\varepsilon_i}$ is $\hipa$,
    and a line of \fixaf\ is regular under $\perp_{\varepsilon_1}$
    iff it is regular under $\perp_{\varepsilon_2}$.
    This yields $\fixafr_{\varepsilon_1} = \fixafr_{\varepsilon_2}$.
    
    Let us take any $h_1,h_2\in W$ with $\xi_0(h_1,h_2)\neq 0$.
    Set $\varepsilon_1 = \xi_0(h_1,h_2)$, $a_i = \gen{h_i + e_0}$ for $i=1,2$,
    and let $\varepsilon_2$ be a nonzero scalar $\neq \varepsilon_1$.
    From \eqref{equ:form1} we directly compute that
    $a_1 \perp_{\varepsilon_1} a_2$ and $a_1 \not\perp_{\varepsilon_2} a_2$.
    This yields that $\perp_{\varepsilon_i}$ cannot be defined in terms of 
    $\fixafr_{\varepsilon_i}$.
  \item[$\biegun\in\hipa$:]
    Let $\omega\notin W$. 
    The form $\xi\restriction{W}$ is a  degenerate symplectic form. 
    Set $Y = \omega^\perp \cap W$, then $e_0,\omega\notin Y$.
    Let $e_1,\dots,e_{n-1}$ be a basis of $Y$; then
    $\left( e_0,\dots,e_{n-1} \right)$ is a basis of $W$ and 
    $\left( \omega,e_0,\dots,e_{n-1} \right)$ is a basis of $\field V$.
    Let $y \in Y$, then
    \begin{cmath}
      \xi(y,\omega) = 0,\qquad \xi(y,e_0) = 0, \qquad\text{and}\qquad \xi(e_0,e_0) = 0.
    \end{cmath}
    Take any two vectors $y_i + \alpha_i e_0 + \beta_i\omega$, 
    ($y_i \in Y$, $i=1,2$) of $\field V$.
    We have
    \begin{multline}\label{equ:form2}
      \xi(y_1 + \alpha_1 e_0 + \beta_1\omega,y_2 + \alpha_2 e_0 + \beta_2\omega)
      = \xi_0(y_1,y_2) + (\alpha_1\beta_2 + \alpha_2\beta_1)\lambda 
      + \beta_1\beta_2\mu,
    \end{multline}
    where 
    $\lambda = \xi(e_0,\omega)\neq 0$, $\mu = \xi(\omega,\omega)\neq 0$, and
    $\xi_0$ is the restriction of $\xi$ to $Y$.
    
    For any scalars $\lambda,\mu\neq 0$ let $\xi_{\mu,\lambda}$ be a bilinear form
    defined on $\field V$ by formula \eqref{equ:form2}.
    Let $M$ be the matrix of $\xi_0$ in the given basis.
    Note that $\xi_0$ is a nondegenerate symplectic form.
    Then
    \begin{cmath}
      M_{\mu,\lambda} = 
      \begin{bmatrix} 
        \mu     & \lambda   & 0     & \cdots    & 0 \\
        \lambda & 0         & 0     & \cdots    & 0 \\
        0       & 0 \\
        \vdots  & \vdots    &       & M \\
        0       & 0
      \end{bmatrix}
    \end{cmath}
    is the matrix of $\xi_{\mu,\lambda}$ in our basis.
    Clearly, $\det(M_{\mu\lambda}) \neq 0$, so $\xi_{\mu,\lambda}$ is nondegenerate.
    Let $\perp_{\mu,\lambda}$ be the conjugacy determined by the form $\xi_{\mu,\lambda}$
    and $\fixafr_{\mu,\lambda}$ be the induced structure of regular points and lines.
    
    Since $\xi_{\mu,\lambda}(\omega,\omega)\neq 0$, the form $\xi_{\mu,\lambda}$
    is not symplectic. Easy computation gives that 
    $\xi_{\mu,\lambda}(y + \alpha e_0,y + \alpha e_0) = 0$
    for each $y \in Y$ and each scalar $\alpha$.
    Consequently, $\hipa$ is the set of points selfconjugate under $\perp_{\mu,\lambda}$.
    From \eqref{equ:form2} we compute
    \begin{ctext}
      $\xi_{\mu,\lambda}(y+\alpha e_0, y_1 + \alpha_1 e_0) = \xi_0(y,y_1)$
      and
      $\xi_{\mu,\lambda}(y_1 + \alpha_1 e_0,y+\alpha e_0 + \beta\omega) = 
      \xi_0(y_1,y) + \lambda \alpha_1\beta$
    \end{ctext}
    for all $y,y_1 \in Y$ and scalars $\alpha,\alpha_1,\beta$.
    This yields that for any nonzero scalars $\lambda,\lambda_1,\mu,\mu_1$ 
    and $q',q''\in\hipa$, $a \notin\hipa$
    we have
    \begin{eqnarray}
      q' \perp_{\mu,\lambda} q'' & \iff & q' \perp_{\mu_1,\lambda_1} q'',
      \\
      q' \perp_{\mu,\lambda} a  & \iff & q' \perp_{\mu_1,\lambda} a.
    \end{eqnarray}

    Let us take any $y_1,y_2 \in Y$ with $\xi_0(y_1,y_2)\neq 0$ and let 
    $\mu_1 = \xi_0(y_1,y_2)$. Write $a_i = \gen{y_i + \omega}$.
    Finally, let $\mu_2\neq 0,\mu_1$ be a scalar.
    Then $\fixafr_{\mu_1,\lambda} = \fixafr_{\mu_2,\lambda}$,
    $a_1 \perp_{\mu_1,\lambda} a_2$, and $a_1 \not\perp_{\mu_2,\lambda} a_2$.
    Consequently, 
    $\perp_{\mu_i,\lambda}$ cannot be defined in $\fixafr_{\mu_i,\lambda}$.
  \end{sentences}
  Since in any case, the conjugacy $\perp$ cannot be defined in terms
  of the geometry of $\fixafr_1$, our proof is complete.
\end{proof}


Gathering together \ref{thm:reg2af} and \ref{lem:reg2horiz} we conclude with
\begin{cor}\label{cor:reg2co-sie-da}
  The structure of the form
  \begin{multline}\nonumber
    \big\langle \text{points of }\fixaf,
      \text{lines of\/ }\fixaf,
      \text{points of\/ }\fixproj,
      \text{lines of\/ }\fixproj,
      \hipa\ (=\text{the horizon of\/ }\fixaf),
      \perp_{\hipa},
      \\
      \mathord{\perp}\cap(\reghipy_1\times\hipa)\big\rangle
  \end{multline}
  is definable in terms of the structure $\fixafr_1$.
\end{cor}
Finally, taking into account \ref{cor:reg:sub} from \ref{cor:reg2co-sie-da}
we obtain
\begin{prop}\label{prop:reg2highreg}
  For each integer $k$ the family $\reghipy_k$ is definable in terms of
  the structure $\fixafr_1$.
  Consequently, the family
  \begin{cmath}
    \big\{ A \colon A \text{ is a subspace of\/ }\fixaf,\, \overline{A}\in\reghipy_k  \big\}
  \end{cmath}
  is definable in $\fixafr_1$ as well.
\end{prop}


%
Then from \ref{thm:not:reg2perp:1} we conclude with
\begin{cor}\label{cor:not:reg2perp:gen}
  Let $1\leq k < \dim({\field V})$ be an integer.
  The underlying metric projective space $(\fixproj,\perp)$
  cannot be defined neither in terms of the structure 
  $\GrasSpace(\reghipy,k) = \struct{\reghipy_k,\reghipy_{k+1},\subset}$
  nor in terms of the structure
  $\PencSpace(\reghipy,k) = \struct{\reghipy_k,\peki_k(\reghipy)}$,
  where $\peki_k(\reghipy)$ is the family of regular $k$-pencils.
\end{cor}

  In view of \ref{cor:not:reg2perp:gen}, the two geometries: of 
  regular subspaces of $(\fixproj,\perp)$ and of the metric projective 
  space $(\fixproj,\perp)$, are distinct.

\subsection{More regular point-line geometry}

Now, let us have a look at the incidence structure $\struct{\reghipy_1,\reghipy_2,\reghipy_3}$.
From \ref{lem:regul3biegun} and \ref{lem:reg:plane} we see that

\begin{fact}\label{fct:isolatedlines}
  The lines through\/ $\biegun$ 
  are either nonregular, when $\biegun\notin\hipa$, or isolated as there are no regular 
  planes containing such lines when $\biegun\in\hipa$.
\end{fact}

This is the reason to investigate a new lineset
\begin{cmath}
  \linesrb := \regafhipy_2\cap\regbhipy_2
\end{cmath}
of regular affine lines not through $\biegun$.
This lineset gives rise to a new geometry
\begin{cmath}
  \fixafrb_1 := \struct{\regbhipy_1, \linesrb, \subset},
\end{cmath}
which is a substructure of $\fixafr_1$ (and of $\fixgras_1$).
This slight difference between $\fixafr_1$ and $\fixafrb_1$ has no impact on the validity of \ref{lem:regtriang}, \ref{cor:reconplanes}
and \ref{thm:reg2af} with $\fixafr_1$ replaced by $\fixafrb_1$. 
The respective proofs for $\fixafrb_1$ become a bit more complex 
but are based on the same ideas. Actually we can state even more:

\begin{thm}\label{thm:regb2af}
  The structures\/ $\fixafr_1$ and\/ $\fixafrb_1$ are mutually definable.
  Consequently, the affine space $\fixaf$, and thus the projective space
  $\fixproj$, is definable in terms of\/ $\fixafrb_1$.
\end{thm}

\subsection{Automorphism group of regular point-line geometry}

In view of \ref{thm:reg2af}, $\Aut(\fixafr_1)$ 
is a subgroup of $\Aut(\fixaf)$.  
Let $f\in\Aut(\fixaf)$ and let $f^\infty$ be its action on the horizon $\hipa$
of \fixaf. 
\\
If $f\in\Aut(\fixafr_1)$ then $f^\infty$ must be an automorphism of the induced metric
projective symplectic geometry on $\hipa$.
\\
Moreover, in view of \ref{lem:reg2horiz}, $f$ must preserve the family of hyperplanes
$\set{[q]\colon q \in\hipa}$.
The following is simple, though quite useful.
\begin{lem}\label{crit:aut:gen}
  Let $f\in\Aut(\fixaf)$. The following conditions are equivalent.
  \begin{sentences}\itemsep-2pt
  \item
    $f \in \Aut(\fixafr_1)$
  \item
    $f^\infty \in \Aut(\struct{\hipa,\perp_{\hipa}})$ and 
    $(f,f^\infty)$ preserves $\perp\cap\; (\reghipy_1 \times \hipa)$.
  \end{sentences}
\end{lem}
\begin{proof}
  Immediate by \ref{cor:reg2co-sie-da}.
\end{proof}

\begin{prop}\label{prop:aut:nob3H}
  Let\/ $W$ be the 
  subspace of\/ $\field V$ with $\hipa = \set{\gen{u}\colon u \in W,u\neq \theta}$
  and let $\xi_0$ be the restriction of $\xi$ to $W$.
  Then, clearly, $\xi_0$ determines $\perp_{\hipa}$.
  Assume that  
  $\biegun\notin\hipa$. Then
  \begin{equation}
    \Aut(\fixafr_1) = \big\{ \varphi\in\Gamma L(W) \colon 
    \varphi \text{ preserves } \perp_{\hipa} \big\}.
  \end{equation}
\end{prop}
\begin{proof}
  Note that \fixaf\ can be presented as the affine space $\AfSpace(W)$ over $W$.
  $\biegun$ is the unique point of \fixaf\ such that each line through it is nonregular
  and thus $\biegun$ remains invariant under automorphisms of $\fixafr_1$.
  One can coordinatize $W$ so as $\biegun$ is the origin of the coordinate system
  and thus each automorphism $\varphi$ of $\fixafr_1$ is a semilinear bijection of $W$.
  Since $\hipa$ is the horizon of \fixaf,
  from  \ref{lem:reg2horiz} we get that $\varphi$ preserves $\perp_{\hipa}$.
  A direct computation based on \eqref{equ:form1} justifies that if 
  $\varphi\in\Gamma L(W)$ preserves the conjugacy defined on $\hipa$ by the symplectic
  form $\xi_0$ then $\varphi$ preserves the class of regular lines.
  This closes our proof.
\end{proof}

The technique used in the proof of \ref{thm:not:reg2perp:1} enables us to 
formulate a more elementary definition of the structure $\fixafr_1$.
\begin{prop}
  Let\/ $W$ be the 
  subspace of\/ $\field V$ with $\hipa = \set{\gen{u}\colon u \in W,u\neq \theta}$
  and let $\xi_0$ be the restriction of $\xi$ to $W$.
  Then, clearly, $\xi_0$ determines $\perp_{\hipa}$.
  Moreover, one can represent $\fixafr_1$ as a line reduct of the affine 
  space $\fixaf = \AfSpace(W)$ and $\hipa$ is the horizon of \fixaf.
  Assume that 
  $\biegun\notin\hipa$; consequently, $\biegun$ is a point of $\fixafr_1$.
  \begin{sentences}
  \item
    Let $\lines_\ast$ be the class of lines in the symplectic polar space 
    determined by the conjugacy $\perp_{\hipa}$ on $\hipa$, let
    $a$ be a point of \fixaf, and $\planes^a_\ast$ be the set of planes
    of the form $a + L$ with $L\in\lines_\ast$.
    Finally, write 
    $\lines^a_\ast = \big\{ L\in\lines\colon 
    L\subset A \text{ for some } A\in\planes^a_\ast \big\}$.
    Then 
    $\linesr = \lines\setminus\lines^{\biegun}_\ast$.
    Consequently, for an arbitrary affine point $a$ we have
    $\fixafr_1 \cong \struct{W,\lines\setminus\lines^a_\ast}$.
  \item
    Let $u,v\in W$. Then $u,v$ lie on a nonregular line iff $\xi_0(u,v)=0$.
  \end{sentences}
\end{prop}

Now, we pass to the case $\biegun\in\hipa$. Let us adopt the coordinate system
as in subs. \ref{subsubsec:b0H} and let $\xi_0$ be the restriction of $\xi$ to $W$.

\begin{prop}\label{prop:aut:b3H}
  The following conditions are equivalent.
  \begin{sentences}\itemsep-2pt
  \item
    $f\in\Aut(\fixafr_1)$
  \item
    There are $\varphi\in \Gamma L(W)$ and a vector $\omega\in W$ such that
    $f(x) = \varphi(x) + \omega$ for each $x \in W$, 
    $\varphi$ preserves $\perp_{\hipa}$,
    and the following holds
    \begin{cmath}
      \xi_0(x,y) = \pi_1(y) \implies
      \xi_0(\varphi(x),\varphi(y)) + \xi_0(\omega,\varphi(y)) = \pi_1(\varphi(y)),
    \end{cmath}
    for all $x,y \in W$, where $\pi_1$ is the projection on the 1st coordinate.
  \end{sentences}
\end{prop}

\begin{proof}
  It is clear that each automorphism $f$ of $\fixafr_1$ is a composition of a semilinear map
  $\varphi$ and a translation on a vector $\omega$.
  In the projective coordinates we can write
  $f([1,x]) = [1,\varphi(x)+\omega]$ and $f([0,y]) = [0,\varphi(y)]$.
  The map $f$ of such a form is an automorphism of $\fixafr_1$ iff it preserves $\perp_{\hipa}$
  and it preserves regular lines. From \ref{crit:aut:gen}, $f$ preserves regular lines
  iff it preserves suitable restriction of the polarity.
  To complete the proof it suffices to note that
  \begin{cmath}
    \xi([1,x],[0,y]) = \pi_1(y) + \xi_0(x,y)
  \end{cmath}
  for all $x,y \in W$.
\end{proof}

Suppose that $\varphi \in GL(W)$ and
$\xi_0(\varphi(x),\varphi(y)) = c \xi_0(x,y)$ for some $c \neq 0$
and all $x,y \in W$ (then, clearly, $\varphi$ preserves $\perp_{\hipa}$).
The conditions of \ref{prop:aut:b3H} yield
$\xi_0(\omega,\varphi(y)) = \pi_1(cy + \varphi(y))$ for each $y \in W$.
In particular, if $\varphi = \id$ we obtain $\omega\perp W$ and thus $\omega\parallel\biegun$.
If $\varphi$ is a homothety $x\mapsto \alpha x$ with $\alpha\neq 0$ then 
$c = \alpha^2$ and the condition of \ref{prop:aut:b3H} is read as 
$\xi_0(\omega,y) = (\alpha+1)\pi_1(y)$ for all $y\in W$; this yields
$\omega \perp \set{y\colon \pi_1(y) = 0}$.

\section{Grassmannians of regular secunda and hyperplanes}\label{sec:grasssechip}

Note that $\Rad(U) = \Rad(U^\perp)$. This yields that
\begin{rem}\label{rem:correl}
  The mapping $\perp$ is a correlation in $\fixproj$ which maps regular 
  subspaces to regular subspaces.
\end{rem}

Now set $n := \dim(\fixV) = \dim(\fixproj) + 1$. 
In this case $\reghipy_k^\perp = \reghipy_{n-k}$ and, clearly, 
$\regafhipy_k^\perp = \regbhipy_{n-k}$ for each $k$, $1\le k\le n$. 
Consider the Grassmannian of regular secunda and hyperplanes
\begin{cmath}
  \fixgras_{n-2} := \GrasSpace(\reghipy,n-2) = \struct{\reghipy_{n-2},\reghipy_{n-1},\subset}.
\end{cmath}
It can be easily seen that $\fixgras_1 \cong \fixgras_{n-2}^\perp$.
So, based on \ref{thm:reg2af} we can state that the dual of the affine space $\fixaf$
can be defined in terms of $\fixgras_{n-2}$. In the dual of $\fixaf$
we can define $\fixproj$ as well as in $\fixaf$. 
These observation can be summarized in the following

\begin{thm}\label{thm:sechip}
  The structure\/ $\fixafr_1$ can be reconstructed in terms of\/ $\fixgras_{n-2}$.
\end{thm}

The case where $n = 4$, i.e. $\dim(\fixproj) = 3$, seems quite interesting. Then
$\hipa$ is a plane and $\biegun\in\hipa$.
In this case the Grassmannian $\fixgras_1$ of regular points and lines  is dual isomorphic to the
Grassmannian $\fixgras_2$ of regular lines and planes, i.e. $\fixgras_1 \cong \fixgras_2^\perp$.
Since $\regafhipy_2^\perp = \regbhipy_2$, we have 
$\fixafr_1^\perp\cong\struct{\regbhipy_2, \reghipy_3, \subset}$. The latter structure
is studied in the next section for arbitrary $n$.

\section{Grassmannians of regular lines and planes}

It is a quite complex, though more or less routine, job to define $\fixafr_1$ 
in terms of $\PencSpace(\reghipy,k)$ or $\GrasSpace(\reghipy,k)$.
In this section we shall discuss one particular case of this problem where $k=2$.
It seems, however, that the techniques used here can be applied
generally. 

Consider the Grassmannian of regular lines and planes
\begin{cmath}
  \fixgras_2 := \GrasSpace(\reghipy,2) = \struct{\reghipy_2,\reghipy_3,\subset}.
\end{cmath}
In view of \ref{prop:reg2highreg} the structure $\fixgras_2$ is definable in $\fixafr_1$.
Note by \ref{fct:isolatedlines} that there are isolated points in $\fixgras_2$
iff $\biegun\in\hipa$. 
When we get rid of these isolated points we get a new structure
\begin{cmath}
  \fixafr_2 := \struct{\regbhipy_2, \reghipy_3, \subset}
\end{cmath}
of the regular  lines not through $\biegun$ and regular planes, a substructure of $\fixgras_2$.
Note that $\fixafr_2$ is the dual of $\fixafr_1$ when $\dim(\fixproj) = 3$.

As we already know the incidence structure $\struct{\reghipy_1,\reghipy_2,\reghipy_3}$
of regular points, lines and planes contains isolated objects: the point $\biegun$,
when $\biegun\notin\hipa$, regular lines on $\hipa$ and regular lines through $\biegun$ when
$\biegun\in\hipa$. So, now we introduce the structure
\begin{cmath}
  \fixafrb_2 := \struct{\linesrb, \reghipy_3, \subset}
\end{cmath}
of the regular affine lines not through $\biegun$ and regular planes, a substructure of $\fixafr_2$.
Note that,  when $\dim(\fixproj) = 3$, we have $\linesrb^\perp=\linesrb$ and thus
$\fixafrb_2$ is the dual of $\fixafrb_1$.

Note an evident (cf. \ref{lem:planeonhipa})
\begin{rem}\label{rem:regplanes:onhor}
  There is no regular plane on $\hipa$.
\end{rem}

This means that the points and lines of $\fixafrb_2$ are respectively lines and
planes of the affine space $\fixaf = \struct{\reghipy_1,\lines}$  obtained from
\fixproj\ by deleting the hyperplane $\hipa$.

Recall a known fact
\begin{fact}\label{fct:gras2pek:proj}
  In the structure 
  $\GrasSpace(\fixproj,2) = \struct{\hipy_2,\hipy_3,\subset}$
  consider a triangle
  with the vertices $L_0,L_1,L_2$ and the sides $A_0,A_1,A_2$ labelled so as
  $L_i\not\subset A_i$ for $i=0,1,2$.
  Then the lines $L_i$ have a common point $p$. Moreover
  \begin{multline}
    \big\{ L\in\hipy_2\colon L\subset A_0 \Land \exists{A\in\hipy_3}\;[L,L_0\subset A] \big\}
    = \big\{ L\in\hipy_2\colon p\in L \subset A_0 \big\} 
    = \\
    =: \pek(p,A_0) \in \peki_2(\fixproj),
  \end{multline}
  where $\peki_2(\fixproj)$ is the set of projective planar pencils of lines of\/ \fixproj.
  Consequently, the relation
  \begin{multline}\label{def:gras2pek:proj}
    \wspolin_{\fixproj}(L_1,L_2,L_3) \iff (\exists{L_0\in\hipy_2})(\exists{A,A_1,A_2,A_3\in\hipy_3})
    \big[ L_0\not\subset A \Land 
    \\
    \bigwedge_{i=1}^{3}( L_i \subset A \Land L_i,L_0 \subset A_i) \big]
  \end{multline}
  defined for arbitrary lines $L_1,L_2,L_3$
  coincides with the collinearity relation in the space $\PencSpace(\fixproj,2)$
  of pencils of lines.
\end{fact}

The following is evident by \ref{fct:plane:srodek} and inspection of possible cases.
\begin{lem}\label{lem:pencils}
  Let $A\in\reghipy_3$ and $p$ be a point on $A$. Then $\srodek(A)$ is an affine point and
  \begin{align*}
    \pek(p,A)\cap\regbhipy_2 &= 
    \begin{cases}
      \emptyset,                              & \text{when } p = \srodek(A), \\
      \pek(p,A)\setminus\set{\LineOn(p, \srodek(A))}, & \text{when } p\neq\srodek(A). \\
    \end{cases}
    \\  
    \pek(p,A)\cap\linesrb &= 
    \begin{cases}
      \emptyset,                              & \text{when } p = \srodek(A), \\
      \pek(p,A)\setminus\set{\LineOn(p, \srodek(A))}, & \text{when } p\neq\srodek(A) \text{ and } p\notin A^\infty, \\
      \pek(p,A)\setminus\set{\LineOn(p, \srodek(A)), A^\infty}, & \text{when } p\in A^\infty, \\
    \end{cases}
  \end{align*}
  Consequently, if $p\notin A^\infty$ (or equivalently $p\notin\hipa$), 
  then $\pek(p,A)\cap\linesr = \pek(p,A)\cap\reghipy_2$.
  Moreover, $\pek(p,A)\cap\reghipy_2 = \pek(p,A)\cap\regbhipy_2$ 
  and $\pek(p,A)\cap\linesr = \pek(p,A)\cap\linesrb$ as $\biegun\notin A^\infty$.
\end{lem}

\begin{lem}\label{lem:reg2peki}
  Let the relation $\wspolin_{\fixafrb_2}$ be defined by the formula 
  \eqref{def:gras2pek:proj} with $\hipy_2$ replaced by $\linesrb$, 
  $\hipy_3$ replaced by $\reghipy_3$,  and let $L_1,L_2,L_3 \in \linesrb$.
  Then the relation $\wspolin_{\fixafrb_2}(L_1,L_2,L_3)$ holds iff 
  $L_1,L_2,L_3\in\pek(p,A)$ for some point $p$ and some $A\in\reghipy_3$.
\end{lem}
\begin{proof}
  \ltor
  Straightforward by \ref{fct:gras2pek:proj}.
  
  \rtol
  Let $A\in\reghipy_3$, $p$ be a point on $A$, and $L_1,L_2,L_3\in\linesrb$ be three lines
  on $A$ through $p$. Set $M := A^\infty$; then $M\in\reghipy_2$. 
  To close the proof we need to find a line $L_0\in\linesrb$ through $p$ 
  but not on $A$ such that
  the planes $A_1 = L_0+L_1$, $A_2 = L_0 + L_2$ and $A_3 = L_0 + L_3$ are all regular.
  There are two cases to consider.
  \begin{sentences}\itemsep-2pt
  \item
    $p$ is an affine point i.e. $p\notin M$.\par
    Note that $M \cap M^\perp = \emptyset$.
    Write $q_i := M \cap L_i$ for $i =1,2,3$.
    Since the $L_i$ are regular, $q_i\notin p^\perp$ and thus  $p\not\perp M$.
    Therefore,
    the intersection $M\cap p^\perp$ is a single point $q_0\neq q_1,q_2,q_3$.
    \par
    Take $x \in M$ with $x\neq q_i$ for $i=0,\dots, 3$ and $y\in M^\perp\cap\hipa + x$
    with $y\neq x$ and $y\notin M^\perp$.
    Then $y\neq\biegun$ and $y \in x + m = \LineOn(x,m)$ 
    for some (uniquely determined by $y$) point $m\in M^\perp$.
    We have $y \notin M$, since otherwise $M = \LineOn(x,y)\ni m$, so $m \in \Rad(M)$.
    \par
    Suppose that $y \in q_i^\perp$ for some $i=1,2,3$.
    Evidently, $q_i \perp M^\perp$ and $q_i\perp q_i$; comparing dimensions we get 
    $q_i^\perp = q_i + M^\perp$. From $y \in q_i^\perp$ we get that 
    $y \in q_i + m_i$ for some $m_i \in M^\perp$. Since $M$ and $M^\perp$ are skew,
    we get $x = q_i$, which is impossible.
    \par
    Suppose that $y \in p^\perp$ and write $K = y + q_0$; then $K \subset p^\perp$.
    For each $i=1,2,3$ the line $K$ crosses $q_i + M^\perp$ in a point $y_i$
    and $y_i \neq y$. Note: $y_i \perp p$.
    There are $m_i \in M^\perp$ such that $y_i \in q_i + m_i$;
    it is seen that $m_i = m$, as otherwise the lines 
    $M$ and $\LineOn(m,m_i)\subset M^\perp$
    are contained in the plane $M + K$ and thus they have a common point.
    Consider another point $y'$ on $m + x$. If there were $y'\in p^\perp$
    we would obtain another point $y'_1$ on $m + q_1$ with $p\perp y'_1$.
    This leads, contradictory, to $q_i\perp p$.
    Without loss of generality we can assume that $y \notin p^\perp$.
    \par
    Take $L_0 := \LineOn(p,y)$; then $L_0\in\linesrb$ as $p$ is an affine point.
    Take $A_i := L_0 + L_i$ for $i =1,2,3$. Then $A_i^\infty = \LineOn(q_i,y)$.
    Since $q_i \not\perp y$, the lines $A_i^\infty$ are regular and thus 
    $A_i\in\reghipy_3$.
  \item
    $p\in M$, $M\neq L_1,L_2,L_3$ i.e. all the $L_i$ are affine lines. 
    \par
    By assumption, $M$ is regular.
    Let $D$ be a plane on $\hipa$ containing the line $M$. 
    By \ref{lem:planeonhipa}  the radical of $D$ is a point, say $u$. The line $L := \LineOn(p,u)$
    is nonregular. Now, in the three space $A+D$ we take a line $L_0$ through $p$ 
    that is not contained in any of the planes $L_i+L$ and is not contained in $p^\perp$.
    This is doable thanks to assumption that there are at least 5 
    lines distinct from $M$ through $p$ on every plane
    containing $M$. 
    Note that $L_0$ is regular by \ref{lem:reg:line}.
    \\
    Let $A_i := L_i+L_0$.
    By \ref{lem:reg:plane} each of the planes $A_i$ is regular as it meets $D$ 
    in a line distinct from $L$ thus, in a regular one.
  \end{sentences}  
  This completes the reasoning.
\end{proof}

\begin{cor}\label{cor:reg2peki}
  Let the relation $\wspolin_{\fixafr_2}$ be defined by the formula 
  \eqref{def:gras2pek:proj} with $\hipy_2$,  $\hipy_3$  replaced by $\regbhipy_2$, 
  $\reghipy_3$ respectively  and let $L_1,L_2,L_3 \in \regbhipy_2$.
  The relation $\wspolin_{\fixafr_2}(L_1,L_2,L_3)$ holds iff 
  $L_1,L_2,L_3\in\pek(p,A)$ for some regular point $p$ and some $A\in\reghipy_3$.
\end{cor}

\begin{proof}
  \ltor
  Clear by \ref{fct:gras2pek:proj}.
    
  \rtol
  Let $A\in\reghipy_3$, $p$ be a point on $A$, and $L_1,L_2,L_3\in\regbhipy_2$ be three lines
  on $A$ through $p$. Set $M := A^\infty$; then $M\in\regbhipy_2$. 
  We have three cases
  \begin{sentences}\itemsep-2pt
  \item\label{cor:reg2peki:i}
    $p$ is an affine point i.e. $p\notin M$,
  \item\label{cor:reg2peki:ii}
    $p\in M$, $M\neq L_1,L_2,L_3$ i.e. all the $L_i$ are affine lines,
  \item\label{cor:reg2peki:iii}
    $p\in M = L_i$ for some $i\in\set{1,2,3}$. 
  \end{sentences}
  Cases \eqref{cor:reg2peki:i} and \eqref{cor:reg2peki:ii} follow directly 
  by \ref{lem:reg2peki} (though in case \eqref{cor:reg2peki:ii} there is
  a simple independent proof here that do not require 6 lines in a projective pencil).
  In the remaining case \eqref{cor:reg2peki:iii} without loss of generality we take $i=1$.
    Let $D$ be a plane through $M$ contained in $\hipa$. By \ref{lem:planeonhipa},
    since $D$ contains a regular
    line $M$, $\Rad(D)$ is a point $q$.
    If there was $p = q$ we would have 
    $p \perp D$, which gives, contradictory,  $p\in\Rad(L_1)$. 
    The unique nonregular line
    through $p$ on $A$ is $K_1:= \LineOn(p,{\srodek(A)})$
    and the unique nonregular line through $p$ on $D$ is $K_2:=\LineOn(p,q)$.
    Write $A_0:= K_1 + K_2$; since $p\perp K_1,K_2$ we have $p \perp A_0$.
    Set $Y := D + A$; then $Y \in\hipy_4$.
    Note that $\rdim(Y)\le 1$; if $\Rad(Y)\neq \emptyset$ then $\Rad(Y) = q$.
    Since, 
    either $Y \subset p^\perp$, which gives,
    contradictory, $p\perp L_1$, or $Y\cap p^\perp$ is a plane, 
    we obtain $Y\cap p^\perp = A_0$.
    \par
    Take $M_3\subset D$ with $p \in M_3 \neq K_2,M$; then $M_3$ is regular.
    For $A_3 = L_3 + M_3$ we have $A_3 \in \reghipy_3$, because $A_3^\infty = M_3$.
    From $p \in A_3,A_0 \subset Y$ we get that $K_0 := A_3 \cap A_0$ is a line.
    \par
    Consider the plane $B = L_2 + K_2$; then $B$ and $A_3$ have a common line $K_4$.
    Let $L_0$ be a line in $\pek(p,A_3)$ distinct from $L_3,K_4,K_0,M_3$.
    Then $L_0\not\subset A_0$ and thus $L_0$ is regular.
    Set $A_2 = L_0 + L_2$; then $p \in A_2^\infty$ and 
    $A_2^\infty = A_2 \cap \hipa = A_2 \cap D \neq K_2$, which gives that $A_2$ is regular.
    Finally, we take $A_1 = L_0 + L_1$. Then $A_1^\infty = M$ and thus 
    $A_1\in\reghipy_3$.
\end{proof}

The family of equivalence classes of the relation $\wspolin_{\fixafrb_2}$ is the set
\begin{cmath}
  \big\{ \pencraf(p,A)
  \colon p \in A \in \reghipy_3\big\} \setminus \big\{ \emptyset \big\},
  \quad\text{where}\quad
  \pencraf(p,A) = \set{L\in\linesrb \colon p \in L \subset A}.
\end{cmath}
So, we get the space of affine regular pencils $\AfPencSpace(\reghipy,2)$ with regular affine lines
not through $\biegun$,
i.e. elements of $\linesrb$, as points and affine regular pencils $\pencraf(p,A)$ as lines.
Note that among pencils $\pencraf(p,A)$ we have proper pencils, those with $p\notin\hipa$,
and parallel pencils, those with $p\in\hipa$.

The family of equivalence classes of the relation $\wspolin_{\fixafr_2}$, is the set
\begin{cmath}
  \big\{ \pencr(p,A)
  \colon p \in A \in \reghipy_3\big\} \setminus \big\{ \emptyset \big\},
  \quad\text{where}\quad
  \pencr(p,A) = \set{L\in\regbhipy_2 \colon p \in L \subset A},
\end{cmath}
and we have the space of regular pencils $\PencSpace(\reghipy,2)$ with regular lines
not through $\biegun$ as points and regular pencils $\pencr(p,A)$ as lines. 
The space of affine regular pencils $\AfPencSpace(\reghipy,2)$ is a substructure of
the space of regular pencils $\PencSpace(\reghipy,2)$ in the sense that points of
$\AfPencSpace(\reghipy,2)$ are points of $\PencSpace(\reghipy,2)$ and, in view
of \ref{lem:pencils}, lines of $\PencSpace(\reghipy,2)$ are a bit "richer".

Loosely speaking, we have proved that the family of regular pencils is definable
in both $\fixafrb_2$ and in $\fixafr_2$ (cf. \cite{grasregul}).

Recall that \ref{lem:regul3biegun} and \ref{lem:reg:plane} together say the following.
\begin{fact}\label{fct:regplanes:przezb}
  Let\/ $\biegun\in\hipa$.
  If $A$ is a plane not contained in\/ $\hipa$ and $\biegun\in A$, then 
  $A$ is not regular (comp. {\upshape \ref{rem:regplanes:onhor}}).
  Consequently, no regular pencil exists that contains a line through
  $\biegun\in\hipa$.
\end{fact}
This means that there are no pencils, neither in $\AfPencSpace(\reghipy,2)$ nor in 
$\PencSpace(\reghipy,2)$, with a vertex $\biegun$ when $\biegun\in\hipa$, though
there are regular affine lines through $\biegun$. 

For points $L_1, L_2, L_3$ of a point-line geometry $X$ we write
\begin{equation}
  \textstyle
  \TRG_X(L_1,L_2,L_3) \iff L_1,L_2,L_3 \text{ are the vertices of a triangle in } X.
\end{equation}
From common projective geometry (cf. \ref{fct:gras2pek:proj}) it follows that
if $\TRG_X(L_1,L_2,L_3)$ holds and $X$ is $\fixgras_2$, $\fixafr_2$ or $\fixafrb_2$  and $p\in L_1,L_2$,
then $p\in L_3$ as well.

We are going to identify points of $\fixafr_1$ with stars of lines in $\fixafr_2$ 
as well as with stars of lines in $\fixafrb_2$. 
The star of  regular lines through a point $p$ is the set
\begin{cmath}
  \starofr(p) = \left\{ L\in\reghipy_2\colon p\in L \right\}
\end{cmath}
and the star of regular affine lines through a point $p$ is
\begin{cmath}
  \starofraf(p) = \left\{ L\in\linesr\colon p\in L \right\}.
\end{cmath}
Note that 
$$\starofr(p) = 
  \begin{cases}
    \starofraf(p), & p\notin\hipa,\\
    \starofraf(p)\cup\set{L\in\reghipy_2\colon p\in L\subset\hipa}, & p\in\hipa.
  \end{cases}
$$
To express the notion of a star of lines purely in terms of the geometry
of $\fixafr_2$ or $\fixafrb_2$ for a given pencil ${\cal Q}= \pencr(a,A)\neq\emptyset$ 
(and respectively for ${\cal Q}= \pencraf(a,A)\neq\emptyset$) we write 
\begin{align*}
  \starof_\Delta({\cal Q})  & := \big\{ L\in\reghipy_2\colon (\exists{L_1,L_2\in{\cal Q}})\;\TRG(L,L_1,L_2) \big\}, \\
  \starof_{\rm L}({\cal Q}) & := \big\{ L\in\reghipy_2\colon (\exists{L',L''\in\starof_\Delta({\cal Q})})\;
                                    \bigl[L' \neq L'' \Land  \wspolin(L,L',L'')\bigr] \big\}, \\
  \starof({\cal Q})         & := {\cal Q} \cup \starof_\Delta({\cal Q}) \cup \starof_{\rm L}({\cal Q}),
\end{align*}
where $\TRG = \TRGafr$ and $\wspolin = \wspolin_{\fixafr_2}$ 
(or respectively $\TRG = \TRGafrb$ and $\wspolin = \wspolin_{\fixafrb_2}$).
The set $\starof_\Delta({\cal Q})$ contains those lines $L$ with $L^\infty\not\perp A^\infty$ while
we need $\starof_{\rm L}({\cal Q})$ for the other lines $L$ with $L^\infty\perp A^\infty$.
It will become more apparent later in \ref{lem:pek2star:prop} and \ref{cor:pek2star:prop}.

Let $\adjac$ be the binary collinearity in $\fixgras_2$, i.e. for 
$L_1,L_2\in\reghipy_2$ we write
\begin{ctext}
  $L_1 \adjac L_2$ iff there is $A\in\reghipy_3$ with $L_1,L_2 \subset A$, and
  \\
  $L_1\not\adjac L_2$ when $L_1 \adjac L_2$ does not hold.
\end{ctext}
Following \ref{lem:regul3biegun} recall that 
$$
  \starofraf(\biegun) = \starofr(\biegun)
  \begin{cases}
    = \emptyset \text{ if } \biegun\notin\hipa;
    &
    \text{there are }L_1,L_2\in\starofr(\biegun) \text{ with } L_1\adjac L_2, L_1\neq L_2,
    \\
    \neq\emptyset \text{ if }\biegun\in\hipa;
    &
    L_1\not\adjac L_2 \text{ for all }L_1,L_2\in\starofr(\biegun).
  \end{cases}
$$

\begin{lem}\label{lem:pek2star:prop}
  Let ${\cal Q} = \pencr(p,A)\neq\emptyset$ with $p\in A\in\reghipy_3$ and $p\notin\hipa$,
  and let $L$ be a regular line not in ${\cal Q}$.
  Write $q := L^\infty$ and $M := A^\infty$.
  
  If $p\in L$ and $q\neq\biegun$, then one of the following holds:
  \begin{sentences}\itemsep-2pt
  \item\label{pek2star:cas1}
    $q \notin M^\perp$.
    Then there are distinct $L_1,L_2\in{\cal Q}$ such that 
    $\TRGafr(L,L_1,L_2)$.
  \item\label{pek2star:cas2}
    $q \in M^\perp$. Then there are distinct regular lines $L'$, $L''$
    through $q$
    such that $\wspolin_{\fixafr_2}(L,L',L'')$ and 
    ${L'}^\infty, {L''}^\infty\notin M^\perp$,
    so for both $L'$ and $L''$ the condition \eqref{pek2star:cas1} holds.
  \end{sentences}
  \par
  Conversely, if a line $L$ satisfies \eqref{pek2star:cas1} or \eqref{pek2star:cas2}
  then $p\in L$.
\end{lem}
\begin{proof}
  Note that $M$ is a regular line and $x_1 := p^\perp \cap M$ is a point, since otherwise $p\perp M$
  and then no line in $\penc(p,A)$ is regular. The line $L$ is regular so, $p\not\perp q$.
  \par
  Let $q \notin M^\perp$. Then $x_2 := M\cap q^\perp$ is a point, for if not then
  $q \perp M$, so $q \in M^\perp$.
  Take distinct $y_1,y_2\in M$ such that $y_1,y_2\neq x_1,x_2$ and set 
  $L_i = \LineOn(p,y_i)$ for $i=1,2$.
  From construction, $y_1,y_2 \not\perp p,q$.
  Then $L_1,L_2\in\pencr(a,A)$ by \ref{lem:reg:line}. 
  In view of \ref{lem:planeonhipa} applied for the plane $M+q$ the lines $M_i=\LineOn(q,y_i)$, $i=1,2$ are regular.
  So, the plane  $p+M_i$ is regular for $i=1,2$, which gives $\TRGafr(L,L_1,L_2)$.
  This completes the proof of \eqref{pek2star:cas1}.
  \par
  Let $q\in M^\perp$.
  The nonregular lines contained in $\hipa$ through $q$ are all
  contained in the hyperplane $q^\perp\cap\hipa$ of $\hipa$ 
  (the assumption $q\neq\biegun$ turns out to be essential here).
  It is impossible to decompose $\hipa$ into the union of three proper subspaces
  $q^\perp\cap\hipa$, $p^\perp\cap\hipa$, and $M^\perp\cap\hipa$, 
  so there is a point $q'\in\hipa$
  with $q'\notin q^\perp, M^\perp, p^\perp$.
  Then the line $K:=\LineOn(q,q')\subset\hipa$ is regular and $p\not\perp K$.
  Let $z := K \cap p^\perp$ and $q''\in M$, $q'' \neq q,q',z$.
  Then $q'' \not\perp p$ and thus the lines $L' := \LineOn(p,q')$ and 
  $L'' := \LineOn(p,q'')$ are regular.
  Write $B = K + p$; then $B\in\reghipy_3$ and, evidently, 
  $L,L',L'' \in \pencr(p,B)$.
  Moreover, $q'' \notin M^\perp$, as $M$ is regular by \ref{lem:reg:plane}.
  Since ${L'}^\infty = q'$ and ${L''}^\infty = q''$, the proof in case
  \eqref{pek2star:cas2} is complete.
  \par
  Now, let $L$ be an arbitrary regular line.
  From \ref{fct:gras2pek:proj} we get that
  \eqref{pek2star:cas1} implies $p\in L$.
  In case when \eqref{pek2star:cas2} holds we get 
  $p\in L',L''$, which directly gives $p\in L$.
\end{proof}

Since $p\notin\hipa$ in \ref{lem:pek2star:prop}, then in view of \ref{lem:pencils}
we have

\begin{cor}\label{cor:pek2star:prop}
  Let ${\cal Q} = \pencraf(p,A)\neq\emptyset$ with $p\in A\in\reghipy_3$ and $p\notin\hipa$,
  and let $L$ be a regular affine line not in ${\cal Q}$.
  Write $q := L^\infty$ and $M := A^\infty$.
  
  If $p\in L$ and $q\neq\biegun$, then one of the following holds:
  \begin{sentences}\itemsep-2pt
  \item\label{cor:pek2star:cas1}
    $q \notin M^\perp$.
    Then there are distinct $L_1,L_2\in{\cal Q}$ such that 
    $\TRGafrb(L,L_1,L_2)$.
  \item\label{cor:pek2star:cas2}
    $q \in M^\perp$. Then there are distinct regular affine lines $L'$, $L''$
    through $p$
    such that $\wspolin_{\fixafrb_2}(L,L',L'')$ and 
    ${L'}^\infty, {L''}^\infty\notin M^\perp$,
    so for both $L'$ and $L''$ the condition \eqref{cor:pek2star:cas1} holds.
  \end{sentences}
  \par
  Conversely, if a line $L$ satisfies \eqref{cor:pek2star:cas1} or \eqref{cor:pek2star:cas2}
  then $p\in L$.
\end{cor}

\begin{lem}\label{lem:pek2star:hor}
  Let ${\cal Q} = \pencraf(p,A)\neq\emptyset$ with $p\in A\in\reghipy_3$,
  $p\in\hipa$ and let $L\notin{\cal Q}$ be a regular affine line.
  Then $p\in L$ iff  there are distinct $L_1, L_2\in{\cal Q}$
  such that $\TRGafrb(L,L_1,L_2)$ holds.
\end{lem}
\begin{proof}
  By \ref{lem:reg:plane} and our assumptions the line $M := A^\infty$ is regular.
  Hence $p\neq\biegun$.
  By \ref{lem:pencils} the projective pencil $\penc(p,A)$ contains exactly two 
  lines that are not in $\pencraf(p,A)$, namely $M$ and $K_0 := \LineOn(p,{\srodek(A)})$.
  
  \ltor  
  We have $L\not\subset\hipa$ as $L$ is affine.
  Since $A+L$ is a projective 3-space, $D := (A+L)\cap\hipa$ is a plane.
  Note that $M\subset D$ so, by \ref{lem:planeonhipa} $q := \Rad(D)$ is 
  a point such that $q\notin M$.
  Set $K_1 = \LineOn(p,q)$. Consider the line $K_2:=(K_0 + L)\cap D$.
  Take two lines $M_1, M_2$ on $D$ through $p$ with $M_1,M_2\neq M,K_1,K_2$.
  So, $M_i$ is regular, and by  \ref{lem:reg:plane} the plane $L + M_i$ is regular
  for $i=1,2$. Now take $L_i := (L + M_i)\cap A$, $i=1,2$. Observe that
  $L_1, L_2\neq K_0$.
  Suppose that $L_i$ is nonregular. Then by \ref{lem:reg:line} we have $p\perp L_i$ 
  which is impossible as the plane $A = L_1+L_2$ is regular.
  Clearly, $L_1,L_2 \in{\cal Q}$ and $\TRGafrb(L,L_1,L_2)$.

  \rtol
  A direct consequence of \ref{fct:gras2pek:proj}.
\end{proof}

\begin{cor}\label{cor:pek2star:hor}
  Let ${\cal Q} = \pencr(p,A)\neq\emptyset$ with $p\in A\in\reghipy_3$, 
  $p\in \hipa$, and let $L\notin{\cal Q}$ be a regular line.
  Then $p\in L$ iff  there are distinct $L_1, L_2\in{\cal Q}$
  such that $\TRGafr(L,L_1,L_2)$ holds.
\end{cor}
\begin{proof}
  \ltor
  There are two cases to consider here: 
  (i) $L\subset\hipa$ and 
  (ii) $L\not\subset\hipa$.
  
  In case (i) it suffices to take any two affine lines $L_1,L_2\in{\cal Q}$
  so, we have two new regular planes $L+L_1$, $L+L_2$ by \ref{lem:reg:plane}
  and $A = L_1+L_2$ which means that $\TRGafr(L,L_1,L_2)$ is valid.
  
  The case (ii) follows from \ref{lem:pek2star:hor}.
  
  \rtol
  A direct consequence of \ref{fct:gras2pek:proj}.  
\end{proof}

Note that when $p\in\hipa$, then in $\starof({\cal Q})$ there is no line $L$
with $L^\infty\perp A^\infty$ as such a line is nonregular. 
Thus $\starof_{\rm L}({\cal Q})=\emptyset$ and 
consequently 
$\starof({\cal Q}) = {\cal Q}\cup\starof_\Delta({\cal Q})$. Taking this into account
and summing up \ref{lem:pek2star:prop} together with \ref{cor:pek2star:hor} and
\ref{cor:pek2star:prop} together with \ref{lem:pek2star:hor} we get

\begin{prop}\label{prop:stars}
  Let $p\in A\in\reghipy_3$.
  If\/ $p\notin\hipa$ and\/ $\biegun\notin\hipa$, or $p\in\hipa$ and\/ $p\neq\biegun$, then
  \begin{cmath} 
    \starof(\pencr(p,A))=\starofr(p)\qquad\text{and}\qquad\starof(\pencraf(p,A))=\starofraf(p).
  \end{cmath} 
  If\/ $p\notin\hipa$ and\/ $\biegun\in\hipa$, then
  \begin{cmath} 
    \starof(\pencr(p,A)) = \starof(\pencraf(p,A)) = 
      \starofr(p)\setminus\{\LineOn(p,\biegun)\} = \starofraf(p)\setminus\{\LineOn(p,\biegun)\}.
  \end{cmath}    
\end{prop}

Roughly speaking, all the points of $\fixproj$ except $\biegun$
can be reconstructed in terms of $\fixafr_2$ as well as in terms of $\fixafrb_2$.
Actually, this is expectable since no plane through improper $\biegun$ is regular
(comp. \ref{rem:regplanes:onhor} and \ref{fct:regplanes:przezb}).

Now we will try to distinguish regular and nonregular points of $\fixproj$ in
terms of $\fixafr_2$ and $\fixafrb_2$. To do that we need a convenient
characterization of binary collinearity (adjacency) in these two structures.
Basically two regular affine lines are adjacent iff they are coplanar and the
plane is regular. Note that if a regular line from $\hipa$ and  a regular affine
line are coplanar then the plane is always regular, while regular lines on
$\hipa$ are never adjacent by \ref{rem:regplanes:onhor}. The next lemma sheds
yet more light on (non)adjacency of regular lines.

\begin{lem}\label{lem:char:notadjac}
  Let $L_1,L_2$ be distinct regular lines through a point $p$.
  Then $L_1\not\adjac L_2$ iff one of the following holds
  \begin{enumerate}[\rm(i)]
  \item
    $p\notin\hipa$ and $L_1^\infty\perp L_2^\infty$;
  \item
    $p\in\hipa$ and \parskip0pt
    \begin{enumerate}[\rm a)]
    \item
      either $L_1,L_2\subset \hipa$,
    \item
      or $L_1,L_2\not\subset \hipa$ and $p\perp (L_1 + L_2)^\infty$
    \end{enumerate}
  \end{enumerate}
\end{lem}
\begin{proof}
  Clearly, $A := L_1 + L_2$ is a plane. Therefore, $L_1\not\adjac L_2$ iff 
  $A$ is not regular. 
  
  In case $p\notin\hipa$ the plane $A$ is affine and, in view of \ref{lem:reg:plane},
  it is nonregular iff $A^\infty = \LineOn({L_1^\infty},{L_2^\infty})$ is nonregular
  i.e. iff $L_1^\infty\perp L_2^\infty$.
  
  In case $p\in\hipa$ there are three possibilities.
  If $L_1,L_2\subset\hipa$, then $A\subset\hipa$, so $A$ is nonregular by \ref{rem:regplanes:onhor}.
  If $L_1\subset\hipa$ and $L_2\not\subset\hipa$, then $A^\infty = L_1$, so $A$ is regular
  by \ref{lem:reg:plane}. 
  Finally, if $L_1,L_2\not\subset\hipa$, then $A$ is nonregular iff 
  $A^\infty\perp A^\infty$, which is equivalent to
  $p\perp A^\infty$.
\end{proof}

\begin{lem}\label{lem:gras2hor}
  Assume that $\dim(\fixV) > 4$.
  Let ${\cal Q} = \pencr(p,A)$ or ${\cal Q} = \pencraf(p,A)$ with 
  ${\cal Q}\neq\emptyset$ and $A\in\reghipy_3$.
  The following conditions are equivalent:
  \begin{enumerate}[\rm(i)]
  \item
    $p\in\hipa$; 
  \item\label{cond:defhipa}
    for all $L,L_1,L_2\in\starof({\cal Q})$ from 
    $L\not\adjac L_1,L_2$ it follows that $L_1\not\adjac L_2$.
  \end{enumerate}
\end{lem}
\begin{proof}
  Let $L\in\starof({\cal Q})$. Then $p\in L$ by \ref{prop:stars}.
  In view of \ref{lem:char:notadjac} it suffices to consider the following cases.
  \begin{sentences}\itemsep-2pt
  \item
    Let $p\notin\hipa$. Set $q := L^\infty$.
    Take $q_1\in q^\perp\cap\hipa$ with $q_1\notin p^\perp$
    and $q_2\in q^\perp\cap\hipa$ with $q_2\notin p^\perp, q_1^\perp$. The later
    is doable when $\dim(\fixV)>4$.
    So, we have $q_i\neq \biegun$, $i=1,2$ for if not we would have $q_1\perp q_2$, a
    contradiction.
    Set $L_i := p + q_i$ for $i=1,2$. 
    Then $L\not\adjac L_1,L_2\in\starof({\cal Q})$ and $L_1 \adjac L_2$.
  \item
    Let $p\in\hipa$, $L\not\subset\hipa$, and $L\not\adjac L_1,L_2\in\starof({\cal Q})$. 
    From \ref{lem:char:notadjac} we get $L_1,L_2\not\subset\hipa$.
    Set $A_i := L + L_i$ and $M_i := A_i^\infty$. From assumption, the $M_i$ are
    nonregular lines on $\hipa$ and therefore $p\perp M_i$; this yields $p\perp M_1 + M_2$.
    If $M_1=M_2$ then $(L_1+L_2)^\infty = M_1$, so $L_1\not\adjac L_2$.
    Assume that $M_1\neq M_2$. Then there is a line $M_0 := (L_1 + L_2)\cap(M_1 + M_2)$.
    Clearly, $M_0 = (L_1 + L_2)^\infty$ and $p\in M_0\perp p$, so $M_0$
    is nonregular. Consequently, $L_1\not\adjac L_2$.
  \item
    Let $p\in\hipa$, $L\subset \hipa$, and $L\not\adjac L_1,L_2\in\starof({\cal Q})$.
    From \ref{lem:char:notadjac} we get $L_1,L_2\subset\hipa$, and 
    then $L_1 \not\adjac L_2$ follows.
  \end{sentences}
  This closes our proof.
\end{proof}

\begin{lem}\label{lem:gras2hor3spx}
  If\/ $\dim(\fixV) = 4$, then condition \eqref{cond:defhipa} in {\upshape \ref{lem:gras2hor}}
  is always valid.
\end{lem}
\begin{proof}
  Let ${\cal Q}$ be as assumed in \ref{lem:gras2hor} and let $L, L_1,L_2\in\starof({\cal Q})=\starofr(p)$
  with  $L\not\adjac L_1,L_2$.
  Set $q := L^\infty$, $q_i := L_i^\infty$ for $i=1,2$.
  By \ref{lem:char:notadjac}, $q\perp q_1,q_2$ which, in view of \ref{lem:reg:plane},
  means that planes $L+L_1$, $L+L_2$ are nonregular. 
  
  If $p=\biegun$, then the plane $L_1+L_2$ is nonregular as no plane through $\biegun$ is regular,
  hence $L_1\not\adjac L_2$.
  
  Now assume that $p\notin\hipa$. According to
  the definition of $\starof({\cal Q})$ there is always 
  a regular plane $A'$ such that $L\subset A'$. Hence $q \neq \biegun$ by \ref{fct:regplanes:przezb}.
  In 3-space $\fixproj$ the subspace $K := q^\perp\cap\hipa$ is a line on $\hipa$. Note that $q_1,q_2\in K$ and also
  $q\in K$ as $q\in\hipa$ and thus $q\in q^\perp$. This means that $K$ is isotropic so, $q_1\perp q_2$.
  From \ref{lem:char:notadjac} we get $L_1\not\adjac L_2$.

  Finally, if $\biegun\neq p\in\hipa$, then we have three cases:
  \begin{sentences}\itemsep-2pt
  \item
    Let $L_1, L_2\subset\hipa$. 
    Clearly $L_1\not\adjac L_2$.
  \item
    Let $L_1\subset\hipa$ and $L_2\not\subset\hipa$. 
    If $L\subset\hipa$, then the plane
    $L+L_2$ is regular by \ref{lem:reg:plane} which contradicts our assumptions that $L\not\adjac L_2$. 
    If $L\not\subset\hipa$, then the plane $L+L_1$ is regular, again a contradiction.
  \item
    Let $L_1, L_2\not\subset\hipa$.
    If $L\subset\hipa$, then both planes $L+L_1$, $L+L_2$ are regular which is impossible by
    our assumptions that $L\not\adjac L_1,L_2$.
    If $L\not\subset\hipa$, then we have two distinct isotropic lines $M_i := (L+L_i)\cap\hipa$, $i=1,2$
    on $\hipa$. Since $p\in M_1, M_2$ and $M_1\cap M_2=\biegun$, a contradiction and the proof is complete.
  \end{sentences}
\end{proof}

As it has been shown in \ref{lem:gras2hor3spx} we need to treat the case $\dim(\fixV) = 4$ 
separately. 
Recall that $\fixafr_2\cong\fixafr_1^\perp$ and
$\fixafrb_2\cong\fixafrb_1^\perp$. So, we can state that the dual of
$\fixaf$ is definable in $\fixafr_2$ and $\fixafrb_2$ by \ref{thm:regb2af} and \ref{thm:reg2af}. 
Hence $\fixproj$ is definable in $\fixafr_2$. Moreover, the horizon $\struct{\hipa,\perp_{\hipa}}$
is also definable, so is $\fixafr_1$.

In \ref{prop:stars} we have defined points of $\fixproj$ and thanks to
\ref{lem:gras2hor} we are able to distinguish
regular and nonregular points, all strictly in languages of $\fixafr_2$
and of $\fixafrb_2$. So, we have $\struct{\hipa,\perp_{\hipa}}$ reconstructed. 
Gathering all together we get

\begin{thm}\label{thm:regl2regp}
    The structure $\fixafr_1$ can be defined in terms of both\/ $\fixafr_2$
    and\/ $\fixafrb_2$.
\end{thm}

\bigskip
\begin{small}
\noindent
Authors' address:
\\
Krzysztof Pra{\.z}mowski,
Mariusz {\.Z}ynel
\\
Institute of Mathematics, University of Bia{\l}ystok
\\
ul. Akademicka 2, 15-267 Bia{\l}ystok, Poland
\\
\verb+krzypraz@math.uwb.edu.pl+,
\verb+mariusz@math.uwb.edu.pl+
\end{small}

\end{document}